\documentclass[11pt]{article}
\usepackage[margin=1in]{geometry}
\usepackage{amsthm}
\usepackage{graphicx}
\usepackage{amsmath}
\usepackage{amssymb}
\usepackage{mathtools}
\usepackage{soul}
\usepackage{stmaryrd}
\usepackage{array}
\usepackage{enumitem}
\usepackage{comment}
\usepackage{subcaption}
\usepackage{caption}
\usepackage{xcolor}
\usepackage{hyperref}
\usepackage{cleveref}
\usepackage{wrapfig}
\usepackage{tabularx}
\usepackage[font=small,skip=4pt]{caption}

\usepackage{algorithm}
\usepackage{algcompatible}
\makeatletter
\AddToHook{env/algorithmic/begin}{\def\@currentcounter{ALG@line}}
\makeatother
\crefalias{ALG@line}{line}

\newtheorem{lemma}{Lemma}
\newtheorem{remark}{Remark}

\DeclareMathOperator*{\argmin}{arg\,min}

\newcommand{\R}{\mathbb{R}}
\newcommand{\bx}{\mathbf{x}}
\newcommand{\bv}{\mathbf{v}}
\newcommand{\by}{\mathbf{y}}
\newcommand{\bb}{\mathbf{b}}
\newcommand{\bg}{\mathbf{g}}
\newcommand{\bu}{\mathbf{u}}
\newcommand{\bbf}{\mathbf{f}}

\newcommand{\MP}{P} %
\newcommand{\MR}{R} %
\newcommand{\SP}{\mathcal{P}} %
\newcommand{\SR}{\mathcal{R}} %
\newcommand{\PoU}{\mathcal{D}} %
\newcommand{\cosa}{\cos\theta}
\newcommand{\sina}{\sin\theta}

\title{Algebraic Multigrid with Overlapping Schwarz Smoothers
and Local Spectral Coarse Grids for Least Squares Problems%
\thanks{Corresponding author: \texttt{southworth@lanl.gov}.}%
\thanks{Funding: BSS and GAW were supported by the Laboratory Directed Research and Development program of Los Alamos National Laboratory, under project number 20240261ER. HAD and ET were supported by the Fusion Futures Programme. As announced by the UK Government in October 2023, Fusion Futures aims to provide holistic support for the development of the fusion sector. Los Alamos National Laboratory report number LA-UR-25-32252.}%
}

\author{%
Ben S. Southworth\thanks{Theoretical Division, Los Alamos National Laboratory, USA.} \and
Hussam Al Daas\thanks{Science and Technology Facilities Council, Scientific Computing Department, Rutherford Appleton Laboratory, UK.} \and
Golo A. Wimmer\footnotemark[3] \and
Ed Threlfall\thanks{UK Atomic Energy Authority, Culham Campus, UK.}%
}

\begin{document}
\maketitle
\allowdisplaybreaks

\begin{abstract}
This paper develops a new algebraic multigrid (AMG) method for sparse least–squares systems of the form $A=G^TG$ motivated by challenging applications in scientific computing where classical AMG methods fail. First we review and relate the use of local spectral problems in distinct fields of literature on AMG, domain decomposition (DD), and multiscale finite elements. We then propose a new approach blending aggregation-based coarsening, overlapping Schwarz smoothers, and locally constructed spectral coarse spaces. By exploiting the factorized structure of $A$, we construct an inexpensive symmetric positive semidefinite splitting that yields local generalized eigenproblems whose solutions define sparse, nonoverlapping coarse basis functions. This enables a fully algebraic and naturally recursive multilevel hierarchy that can either coarsen slowly to achieve AMG-like operator complexities, or coarsen aggressively-with correspondingly larger local spectral problems—to ensure robustness on problems that cannot be solved by existing AMG methods. The method requires no geometric information, avoids global eigenvalue solves, and maintains efficient parallelizable setup through localized operations. Numerical experiments demonstrate that the proposed least-squares AMG-DD method achieves convergence rates independent of anisotropy on rotated diffusion problems and remains scalable with problem size, while for small amounts of anisotropy we obtain convergence and operator complexities comparable with classical AMG methods. Most notably, for extremely anisotropic heat conduction operators arising in magnetic confinement fusion, where AMG and smoothed aggregation fail to reduce the residual even marginally, our method provides robust and efficient convergence across many orders of magnitude in anisotropy strength.
\end{abstract}

\section{Introduction}

This paper focuses on developing robust algebraic multilevel solvers, combining ideas from domain decomposition and overlapping Schwarz methods with algebraic multigrid (AMG), for large-scale sparse linear systems of the least squares form
\begin{equation}\label{eq:Ax=b}
A \bx\coloneqq  G^TG\bx = \bb,
\end{equation}
where $G\in\R^{m\times n}$ is sparse, and $A\in\R^{n\times n}$ is symmetric positive definite (SPD). We are particularly targeting methods that can achieve AMG-like operator complexities and scalability when feasible (e.g., isotropic Poisson), while providing robustness on problems that are too hard for existing pointwise AMG methods. The least squares formulation is particularly relevant for matrices of the form $D_1 + B^T D_2^{-1} B$, where $D_1\in\R^{n\times n},D_2\in\R^{m\times m} $ are diagonal and positive, ensuring they have a real-valued square root, and $B\in\R^{m\times n}$ is sparse. We can then define 
\[ 
    G^T \coloneqq \begin{pmatrix} \sqrt{D_1} & B^T\sqrt{D_2^{-1}} \end{pmatrix},
\]
arriving at $A = D_1 + B^T D_2^{-1} B = G^TG$. More generally, $D_1$ and $D_2$ can also be of non-diagonal form, as long as they admit a sparse factorization $D_1 = G_1^TG_1$, $D_2 = G_2^T G_2$. 

Efficient and scalable solvers for large-scale linear systems are essential in scientific computing. The particular structure we discuss above arises in a variety of applications, including weighted least-squares problems, constrained optimization, least-squares finite elements discretisations, mixed finite element discretizations, hybridized formulations of partial differential equations, and Schur complement approximation. One notable example is the Schur complements of symmetric block $2\times 2$ systems
\begin{equation}
    \begin{bmatrix}
        M_{1} & B^T \\ B & -M_{2}
    \end{bmatrix}
    \begin{bmatrix}
        \bx \\ \by 
    \end{bmatrix}
    = \begin{bmatrix}
        \bbf \\ \bg
    \end{bmatrix},
\end{equation}
where $M_{1},M_{2}$ are mass matrices that are either directly invertible (e.g., due to mass lumping or element-local mass matrices as in DG methods) or spectrally equivalent to diagonal matrices, and the efficacy of standard block preconditioning approximations is fully defined by the approximation of the Schur complement \cite{southworth2020fixed}
\[
    S \coloneqq M_{1} + B^TM_{2}^{-1}B \sim D_{1} + B^TD_{2}^{-1}B.
\]
Elliptic and parabolic operators also often admit such a structure, where the spatial operator can be factored, e.g. $-\Delta = (\nabla)^T\nabla$. 

Algebraic multigrid (AMG) methods~\cite{Xu.2017} and domain decomposition (DD) techniques~\cite{nataf15} are among the most successful strategies for preconditioning large sparse SPD systems. When applicable, AMG is among the fastest and most scalable linear solver, in terms of both problem size and parallel efficiency. AMG methods are attractive for their algebraic nature and scalability, and are originally designed for, and most effective on, scalar equations arising from elliptic or parabolic PDEs. However, there are many systems arising in numerical PDEs that remain challenging or infeasible for current AMG techniques, some of which are of the form in \eqref{eq:Ax=b}. Recently, several robust multigrid solvers have been proposed to solve a variety of elliptic PDEs by using techniques from exterior calculus analysis, which showed that robustness can only be achieved by using DD- or patch-based smoothers, e.g. \cite{brubeck2024multigrid,farrell2021pcpatch}. Spectral DD methods, on the other hand, offer robustness through the use of overlapping Schwarz-based smoothers and coarse spaces constructed from local generalized eigenvalue problems, e.g. \cite{Galvis.2010,Efendiev.2012,Scheichl.2011,Nataf.2011,Spillane.2014,Spillane.2025},
allowing them to adaptively bound the condition number of the preconditioned matrix. However, their setup cost can be significant, particularly since large subdomains are typically used. Although most DD methods are two-level, recently there have been multilevel extensions, e.g. \cite{Heinlein.2023,Gander.2020,Bastian.2022,Daas.2021}.

In this work, we propose a new multilevel least-squares AMG-DD (LS-AMG-DD) method that combines ideas from spectral domain decomposition and AMG, tailored specifically for SPD matrices of the form \eqref{eq:Ax=b}. The overarching objective is a method that can naturally coarsen slowly and construct a sparse multilevel hierarchy with low operator complexities like classical AMG methods, or if necessary coarsen aggressively and use large local spectral problems to define robust coarse grids for problems that cannot be solved with existing AMG approaches and slow coarsening. This balance allows for the construction of a multilevel hierarchy in a recursive and straightforward manner, without the need for expensive global operations or large-scale eigenvalue problems unless absolutely necessary. In both cases, we utilize overlapping Schwarz smoothers based on aggregation of algebraic degrees of freedom \cite{Scheichl.2007zs,frommer2014adaptive} to overcome limitations inherent to pointwise smoothers. The LS-AMG-DD method is fully algebraic and exploits the sparsity and multiplication structure of $A$ to construct local SPD splitting matrices efficiently. The resulting preconditioner is scalable and robust, with a setup phase that is both parallelizable and efficient.

The remainder of the paper and contributions are organized as follows.
\begin{itemize}
    \item Section~\ref{sec:background} reviews relevant background on domain decomposition and AMG methods. In particular, in \Cref{sec:background:spectral} we provide a detailed discussion on spectral coarse grids in AMG and DD methods, providing direct relations between the typically independent fields of work on AMG, DD, and generalized multiscale finite elements (GMsFEMs). \item Section~\ref{sec:amg-dd} develop a new multilevel preconditioner (LS–AMG–DD) for least-squares operators $A=G^TG$ that couples algebraic aggregation, local spectral coarse spaces, and overlapping Schwarz smoothing. We derive a simple SPSD splitting based on the factorization 
    $A=G^TG$ enabling inexpensive local generalized eigenproblems, and show how to propagate the least-squares structure to coarse levels, yielding a fully algebraic multilevel method.
    \item Section~\ref{sec:numerics} provides numerical experiments demonstrating that the method achieves AMG-like performance on classical elliptic problems and robust convergence on extreme anisotropy in magnetic confinement fusion, where classical AMG fails entirely.
    \item Section~\ref{sec:conclusion} concludes with a discussion of future directions.
\end{itemize}

\section{Background}
\label{sec:background}

Because we are pulling from distinct fields of both AMG and DD, we will first set the notation. We define $\MP$ to be the multigrid interpolation operator (and in future work, $\MR$ the multigrid restriction operator if $\MR\neq\MP^T$). For overlapping Schwarz relaxation and subdomains, we will assume a partitioning of nodes in the sparsity graph of $A$ into a set of nonoverlapping aggregates $\{\omega_i\}_{i=1}^{n_c}$. Here we are working in the fully discrete setting, so DOFs and aggregates always refer to matrix-vector entries (as opposed to, e.g. mesh or physical domain information) unless otherwise notes. For the $i$th aggregate, we define the following vectors of node indices:
\begin{align*}
    \omega_i &: \textnormal{ set of nodes in $i$th nonoverlapping aggregate}\\
    \Gamma_i &: \textnormal{ set of neighboring nodes connected to $\omega_i$}\\
    \Omega_i &: \textnormal{ $i$th overlapping aggregate}, \omega_i\cup\Gamma_i\\
    \Delta_i &: \textnormal{ all nodes not in $i$th overlapping aggregate}, \{1,\ldots,n\} \backslash \Omega_i
\end{align*}
With the $i$th aggregate, we define the Schwarz restriction matrices to the different sets via appropriate subsets by row of the identity matrix: 
\begin{equation}\label{eq:SR}
    \SR_{\omega_{i}} \coloneqq I(\omega_i,:), \hspace{3ex}
    \SR_{\Gamma_{i}} \coloneqq I(\Gamma_i,:), \hspace{3ex} 
    \SR_i  \coloneqq I(\Omega_i,:), \hspace{3ex} 
    \SR_{\Delta_{i}} \coloneqq I(\Delta_i,:).
\end{equation}
Similarly, a permutation matrix for defining and analyzing the overlapping smoothers will also be defined later and denoted by $\SP$, where $\SR$ and $\SP$ are differentiated from multigrid transfer operators by the script notation.

\subsection{Algebraic multigrid}
\label{sec:background:amg}

AMG is a well-known multilevel iterative method to solve large sparse matrix equations, and consists of two properties: (i) relaxation, and (ii) coarse-grid correction. Relaxation is some approximation $M^{-1}\approx A^{-1}$ which is computationally cheap to apply, and is used as a residual correction 
\[
    \bx \mapsfrom \bx + M^{-1}(\bb - A\bx).
\]
Most AMG methods use simple pointwise smoothers such as variations in Jacobi or Gauss-Seidel. Such pointwise smoothers almost universally attenuate error associated with large eigenvalues. 

Coarse-grid correction is based on a Galerkin projection of the residual to a coarse space, wherein a coarse-grid operator is inverted, and the result interpolated back as a fine-grid correction. Given interpolation operator $P\in\R^{n\times n_c}$ for $n_c\ll n$, coarse-grid correction takes the form
\[
    \bx \mapsfrom \bx + P(P^TAP)^{-1}P^T(\bb - A\bx).
\]
Here $A_c\coloneq P^TAP$ is the Galerkin coarse-grid operator. The associated error propagation operator $I - \Pi_A \coloneqq I-P(P^TAP)^{-1}P^TA$ is an $A$-orthogonal projection onto the range of $P$. Combining pre- and/or post-relaxation with coarse-grid correction yields a two-level method; multilevel is achieved by recursively approximating $(P^TAP)^{-1}$ in an analogous manner until the coarsest grid is small enough to easily solve directly. 

For an effective AMG method, it is critical that relaxation and coarse-grid correction attenuate complementary components of the error. Under the assumption of pointwise smoothers, coarse-grid correction is thus typically responsible for attenuating error associated with small eigenvalues. The fundamental components of designing an AMG algorithm are then coarsening the total size of the problem, and proceeding to define a sparse interpolation operator around this coarsening. 
Due to the original target of multigrid and AMG methods being elliptic or parabolic problems, it is implicitly assumed in many forms of AMG that \emph{error associated with small eigenvalues varies slowly in the direction of strong connections of the matrix.} When such a property holds, AMG is often one of the fastest and most efficient (in parallel) algebraic solvers. In contrast, if a problem has fundamental modes that do not follow this behavior, many forms of AMG quickly break down. 

Much of the AMG literature can broadly be classified as either coarsening using a coarse-fine (CF) partitioning of points, or an aggregation of fine level points into a non-overlapping partition of unity. Here we will focus on the aggregation class of methods, due to its conceptual similarity to (algebraic) DD methods. In aggregation-based AMG, aggregates are constructed based on matrix entries. Typically this is done by first constructing a strength-of-connection (SOC) matrix trying to identify a subset of the matrix graph corresponding to ``strong'' connections $\mathcal{S}$. A typical SOC measure is  along the lines of $a_{ij} \in\mathcal{S}$ if $|a_{ij}| \geq \theta \max_{k\neq i} |a_{ik}|$, where $0 \leq \theta \leq 1$ is some user-specified tolerance. The ``standard'' aggregation routine from \cite{Vanek.1996} consists of multiple passes. First, aggregates are formed as greedy disjoint strongly coupled neighborhoods. Then, unaggregated points are grouped into existing aggregates for which they share some strong connection. There is a final cleanup phase for any remaining unaggregated points, but this is typically not necessary. Many routines have been developed for aggregation, but as we will discuss later, too much guiding of the aggregation process can undermine our proposed interpolation operators. 

The classical smoothed aggregation (SA) method forms a tentative interpolation operator defined to be constant over the nonoverlapping aggregates \cite{Vanek.1996}. The sparsity pattern is then expanded and the range of interpolation improved by applying one or two smoothing iterations, e.g. weighted Jacobi, to the tentative interpolation operator. This is effectively approximating geometric bilinear interpolation in the algebraic setting by using smoothing applied to a local constant vector to expand and smooth the basis function support. Recent methods have considered energy minimization over fixed sparsity patterns, e.g. \cite{olson2011general,manteuffel2017root}. SA and energy-min aim to capture the global near nullspace in the range of interpolation via defining columns as locally smooth vectors over aggregates with some expanded support.

\subsection{Algebraic domain decomposition smoothers}
\label{sec:background:dd}

Here we discuss the construction of a purely algebraic overlapping domain decomposition smoother based on aggregates of matrix entries or nodes. Given a set of aggregates $\omega_1, \ldots, \omega_{n_c}$ covering the set of nodes in the sparsity graph of $A$, we add to each aggregate the neighboring nodes in the sparsity graph of $A$ to obtain a set of overlapping aggregates $\Omega_1, \ldots, \Omega_{n_c}$. Recall we define the set of neighboring nodes to $\omega_i$ as $\Gamma_i$, and the resulting overlapping aggregates $\{ \Omega_i\}_{i=1}^{n_c}$ are given by $ \Omega_i \coloneqq \omega_i\cup \Gamma_i$, with respective complements $\{ \Delta_i\}_{i=1}^{n_c}$ for $\Delta_i \coloneqq \{1,\ldots,n\} \backslash \Omega_i$.
We define the $i$th aggregate matrix as the principal submatrix of $A$,
\[
    A_{i} \coloneqq A(\Omega_i,\Omega_i).
\]
We also attach to each overlapping aggregate a partition-of-unity (PoU) matrix $\PoU_i$ such that:
\begin{equation}
    \label{eq:PoU}
    \sum_{i=1}^{n_c} \SR_i^T \PoU_i \SR_i = I.
\end{equation}
While any PoU matrices satisfying \cref{eq:PoU} are acceptable, it is usual to consider diagonal matrices. In this work, we consider the simple diagonal Boolean PoU matrices, where a diagonal value is set to 1 if the corresponding node is associated with a nonoverlapping aggregate and 0 otherwise,
\begin{equation}\label{eq:pou2}
        [\PoU_i]_{\ell\ell} = \begin{cases} 1 & \ell \in \omega_i, \\ 0 & \ell\in\Gamma_i . \end{cases}
\end{equation}
As we will see later, this is advantageous to reduce fill-in of Galerkin coarse-grid matrices.

Using the Schwarz restriction matrices, we can define a permutation matrix 
\begin{equation}
    \label{eq:permutation_matrix}
    \SP_i = I([\SR_{\omega_{i}},\SR_{\Gamma_{i}}, \SR_{\Delta_{i}}], :),
\end{equation}
where $[\SR_{\omega_{i}},\SR_{\Gamma_{i}}, \SR_{\Delta_{i}}]$ is a vector reordering of the rows. 
This permutation matrix allows to reorder the matrix global $A$ to a three-by-three block tridiagonal form
\begin{equation}
    \label{eq:ordered_A}
    \SP_i A \SP_i^T = \begin{pmatrix}
        \SR_{\omega_{i}}A\SR_{\omega_{i}}       & \SR_{\omega_{i}} A \SR_{\Gamma_{i}}     & \\
        \SR_{\Gamma_{i}}A\SR_{\omega_{i}}  & \SR_{\Gamma_{i}} A \SR_{\Gamma_{i}} & \SR_{\Gamma_{i}} A \SR_{\Delta_{i}}\\
        & R_{\Delta_{i}}A \SR_{\Gamma_{i}} & \SR_{\Delta_{i}}A\SR_{\Delta_{i}}
    \end{pmatrix},
\end{equation}
corresponding to nodes in aggregate $i$ ($\omega_i$), neighboring nodes directly connected to aggregate $i$ ($\Gamma_i$), and all remaining nodes ($\Delta_i$), respectively.

Domain decomposition smoothing operators can be seen as a combination of local solves on overlapping aggregates. Local smoothing involves three steps, namely, restriction, correction, and prolongation (in this sense similar to the larger mulitigrid solver). These steps can be formulated using the matrices we defined above.
The restriction operation restricts a given global vector to the local overlapping aggregate.
The correction step corresponds to applying the inverse of the local matrix $A_{i}$. Note that here we only consider $A_{i}$ which is a principal submatrix of the global matrix $A$, but other local matrices can be considered, e.g. \cite{Gander.2020}. The prolongation operation takes a correction vector defined on the overlapping aggregate $\Omega_i$ and expands it globally by setting values outside the overlapping aggregate (that is, on $\Delta_i)$ to zero. Before and/or after the local solve, the local vector can be weighted by using the partition of unity $\PoU_i$. As we have defined it, this simply corresponds to restricting the local vector to only be nonzero in the nonoverlapping aggregate $\omega_i$. 

To define a domain decomposition smoother, a rule needs to be set on the ordering of these local operations as well as the relation with the global residual. Here we consider the restricted additive Schwarz (RAS) smoother 
\begin{equation}
    \label{eq:ras}
    \bx_{k+1} = \bx_k + M_{RAS}^{-1} (\bb-A\bx_k) \coloneqq \bx_k + \sum_{i=1}^{n_c} \SR_i^T \PoU_i A_{i}^{-1} \SR_{i} (\bb-A\bx_k),
\end{equation}
where we restrict the global residual to overlapping aggregates, invert $A_i$ on each overlapping aggregate simultaneously, use the PoU to define the local correction vectors to be nonzero only on nonoverlapping aggregates $\omega_i$ (and zero on $\Gamma_i$), and last interpolate the result back to the global space. The weighting of local vectors by PoU ensures that each global DoF only gets a single update from the additive residual correction in \eqref{eq:ras}.

This paper is focused on SPD problems and the multilevel preconditioners are typically used as preconditioners for Krylov methods. It is desirable to use MINRES or conjugate gradient (CG), both of which require SPD preconditioners, but the RAS smoother is inherently nonsymmetric. Recall that Gauss-Seidel has a (nonsymmetric) forward and backward pass formulation that are adjoints of each other, but when used as a pre- and post-smoother one arrives at a symmetric preconditioner. We do a similar thing with RAS, defining the RAS-transpose (RAS-T) smoother as
\begin{equation}
    \label{eq:ras-t}
    \bx_{k+1} = \bx_k + M_{RAS-T}^{-1} (\bb-A\bx_k) \coloneqq \bx_k + \sum_{i=1}^{n_c} \SR_i^T  A_{i}^{-1}\PoU_i \SR_{i} (\bb-A\bx_k),
\end{equation}
where by swapping the order of $\PoU_i$ and $A_i$, $M_{RAS-T}^{-1}= M_{RAS}^{-T}$. 
One can also consider a multiplicative Schwarz (MS) smoother, but to avoid overly optimistic results associated with ordering that is not amenable in parallel, we do not consider such smoothers.
\subsection{Spectral coarse grids}
\label{sec:background:spectral}

For sparse SPD problems, particularly arising from the discretization of elliptic operators, the concept of energy minimization is fundamental to multigrid methods, e.g. \cite{olson2011general,manteuffel2017root}, where energy is defined by the matrix-induced energy norm $\|\bu\|_A^2 = \langle A\bu,\bu\rangle$. Indeed, error and convergence is typically measured in the $A$-norm, and thus coarse-grid correction is expected to capture ``algebraically smooth'' error, that is, error with a small energy that is not effectively attenuated by pointwise relaxation. For many elliptic PDEs and discretizations, algebraically smooth error is also geometrically smooth, which provides the bridge between algebraic and geometric MG methods. 

Convergence of two-level algebraic preconditioners for SPD matrices is relatively well understood in the $A$-norm \cite{Falgout.2004,Brandt.2011,Xu.2017}. For (potentially nonsymmetric) smoother $M\approx A$, define the symmetrized smoother $\widetilde{M}\coloneqq M(M + M^T-A)^{-1}M^T$, where $I - \widetilde{M}^{-1}A = (I - M^{-T}A)(I - M^{-1}A)$. Note, for SPD $A$ and arbitrary nonsymmetric relaxation $M$ (e.g., the lower triangular part of $A$), this choice of $\widetilde{M}$ ensures that the two-level preconditioner is SPD in the $A$-inner product. Let $\{\bv_\ell\}_{\ell=1}^n$ be the generalized eigenvectors associated with the generalized eigenvalue problem 
\begin{equation}\label{eq:gevp}
    AV = \widetilde{M}V\Lambda,
\end{equation}
with eigenvalues $0<\lambda_1\leq ... \leq \lambda_n$. Then the optimal two-grid transfer operator of size $n_c\ll n$ with respect to error propagation in the $A$-norm is given by defining columns of $P_\#$ as the smallest $n_c$ generalized eigenvectors,
\[
    P_\# \coloneqq \begin{bmatrix} \bv_1 & ... & \bv_{n_c}\end{bmatrix}.
\]
In the typical setting of pointwise smoothers used in AMG, $\widetilde{M}$ is spectrally equivalent to the diagonal $D$ of $A$, and one can consider the simpler eigenvalue problem $AV = DV\Lambda$. 

The optimal interpolation operator can be defined for arbitrary relaxation operator $\widetilde{M}$. However, $P_{\#}$ itself is typically dense and also too expensive to compute in practice. For practical algorithms, a natural idea is to construct $P\approx P_{\#}$ through local spectral approximations. This lines up with a broad field of literature that uses local spectral coarse grids, e.g. \cite{Chartier.2003,Chartier.2007,Lashuk.2008,Efendiev.2013,Spillane.2025,al2019class,Brezina.2011,Chung.2018,Galvis.2010,Efendiev.2012,Scheichl.2011,Nataf.2011,Spillane.2025,Efendiev.2010,Efendiev.2012,Babuska.2011}. Most methods of this class are built on a symmetric positive semi-definite (SPSD) splitting of the problem. Let $\{\Omega_i\}_{i=1}^{n_c}$ denote \emph{overlapping} aggregates that cover the full domain $\Omega$, and let $\SR_i$ denote a binary restriction by value, $\SR_i: \Omega\mapsto\Omega_i$. We define an SPSD matrix splitting to take the form
\begin{equation}\label{eq:spsd}
    A = \sum_i \SR_i^T\widetilde{A}_i\SR_i,
\end{equation}
where $\widetilde{A}_i\in\mathbb{R}^{|\Omega_i|\times|\Omega_i|}$ is a local SPSD submatrix (typically \emph{not} a principal submatrix of $A$).

In addition to the matrix splitting, an additional fundamental component of defining interpolation through local spectral problems is the sparsity pattern of interpolation (i.e. nonzeo entries), and support of global information considered in local spectral problems. We can encompass these choices in a quite general manner by partitioning the discrete problem into a set of nonoverlapping aggregates $\Omega = \cup_i \omega_i$, and overlapping aggregates, where the $i$th overlapping aggregate $\Omega_i = \omega_i \cup \Gamma_i$. We will also refer to nonoverlapping aggregates $\{\omega_i\}$ as \emph{interior} DOFs and the overlapping set $\Gamma_i = \Omega_i\backslash\omega_i$ as \emph{interface} DOFs. Broadly, one can then consider local generalized eigenvalue problems that take the form
\begin{equation}\label{eq:local-gevp}
    \begin{bmatrix} \widetilde{A}_{\omega_i\omega_i} & \widetilde{A}_{\omega_i\Gamma_i} \\
    \widetilde{A}_{\Gamma_i\omega_i} &
    \widetilde{A}_{\Gamma_i\Gamma_i}
    \end{bmatrix}
    \begin{bmatrix} \bv_{\omega_i} \\ \bv_{\Gamma_i} \end{bmatrix} = 
    \lambda    
    \begin{bmatrix} \widetilde{M}_{\omega_i\omega_i} & \widetilde{M}_{\omega_i\Gamma_i} \\
    \widetilde{M}_{\Gamma_i\omega_i} &
    \widetilde{M}_{\Gamma_i\Gamma_i}
    \end{bmatrix}
    \begin{bmatrix} \bv_{\omega_i} \\ \bv_{\Gamma_i} \end{bmatrix},
\end{equation}
for some appropriate SSPD $\widetilde{A}_i,\widetilde{M}_i$, and define interpolation based on local generalized eigenvectors associated with small eigenvalues. This can be motivated by the fact that globally low-energy modes are bounded below by local ones, 
\begin{equation}
    0 \leq \bv^T\SR_i^T\widetilde{A}_i\SR_i\bv \leq \bv^TA\bv
\end{equation}
for all $\bv$, which follows from \eqref{eq:spsd}. To this end, the local spectral framework has a strong relation to energy minimization in general. Let $V_i$ denote the space associated with overlapping aggregate $\Omega_i$. Constructing a space spanned by the smallest eigenvectors of local generalized eigenvalue problems corresponds to a local generalized Rayleigh quotient minimization,
\begin{equation}\label{eq:rayleigh}
    \min_\lambda \widetilde{A}_i\bv = \lambda \widetilde{M}_i\bv \quad\Longrightarrow\quad
    \argmin_{\bv\in V_i} \frac{\langle \widetilde{A}_i\bv,\bv\rangle}{\langle \widetilde{M}_i\bv,\bv\rangle}
    = \argmin_{\bv\in V_i,\|\bv\|_{\widetilde{M}_i}=1} \|\bv\|_{\widetilde{A}_i}^2.
\end{equation}
Here, we are using $\widetilde{A}_i$ from the SPSD splitting to provide a local approximation of energy, and minimizing the local energy normalized with respect to a local smoother $\widetilde{M}_i$. 

In the following sections we show how the principal families of spectral coarse-space methods -- spectral AMGe, GMsFEM, GenEO, and multilevel algebraic Schwarz -- appeal to this same local energy-minimization principle, with particular choices of nonoverlapping and overlapping aggregates, local approximations $\widetilde{A}_i$ and $\widetilde{M}_i$, and potential harmonic (energy minimizing) extensions from the interior (interface) to the interface (interior). Note that this is a concrete realization of the variational
and energy--minimization principles formalized in \cite{Xu.2017}, where AMG methods are analyzed around (i) a \emph{stable decomposition} of the fine space into local subspaces, and (ii) an \emph{approximation property} that bounds the energy of the error after coarse--grid correction. There they assume SPSD splittings of $A$ and $\widetilde{M}$ to define a local energy pair $(A_i,\widetilde{M}_i)$ on subspace $V_i$. The smallest eigenmodes from the resulting local generalized eigenproblem are used to approximate the globally optimal subspace.

The discussion that follows is briefly summarized in \Cref{tab:method_comparison}, with description of the proposed LS-AMG-DD method that will be described in \Cref{sec:amg-dd}.

\begin{table}[!htb]
\small
\centering
\caption{Comparison of major spectral coarse-grid methodologies. 
Here $\omega_i$ denotes nonoverlapping interior DOFs, 
$\Gamma_i$ the interface/overlap, 
and $\Omega_i = \omega_i \cup \Gamma_i$.}
\label{tab:method_comparison}
\begin{tabular}{>{\raggedright\arraybackslash}p{3.3cm} >{\raggedright\arraybackslash}p{4cm} >{\raggedright\arraybackslash}p{4cm}}
\hline
\textbf{Local Domain} &
\textbf{Spectral Problem} &
\textbf{Coarse Basis} \\ \hline\hline

\multicolumn{3}{c}{\textbf{AMGe}} \\\hline
Agglomerated elements (DOFs on $\Omega_i$)
& Eigenproblem on 
$A_{\Omega_i}$ or reduced Schur complement on $\omega_i$ 
with diagonal $\widetilde{M}_i$
& Basis vectors supported on $\omega_i$ or $\Omega_i$;
often smoothed or harmonically extended; 
typically small overlap \\

\multicolumn{3}{c}{\textbf{GMsFEM}} \\\hline
Oversampled coarse neighborhoods (often large $\Omega_i$)
& Weighted Rayleigh quotient on PDE bilinear form;
mass matrix on right-hand side
& Local low-energy modes extended via energy minimizing 
harmonic extension; PoU used for stitching;
good approximation power but high setup cost \\

\multicolumn{3}{c}{\textbf{GenEO (DD)}} \\\hline
Overlapping Schwarz subdomains 
$\Omega_i = \omega_i \cup \Gamma_i$
& Interface eigenproblem on $\Gamma_i$ with 
$A_{\Gamma_i\Gamma_i}$ on the RHS
& Coarse basis from PoU-weighted interface modes
with harmonic extension into $\omega_i$;
robust to heterogeneity; large subdomains \\

\multicolumn{3}{c}{\textbf{LS-AMG-DD (this work)}} \\\hline
Aggregates of algebraic DOFs with one-layer overlap
& Local GEP 
$\PoU_i A_i \PoU_i\, u = \lambda\, \widetilde{A}_i u$ 
using SPSD splitting derived from $A = G^T G$
& Nonoverlapping coarse basis from restricted local modes;
block-diagonal $P$ maintaining sparsity;
naturally multilevel via $G_{\ell+1} = G_\ell P_\ell$ \\

\hline
\end{tabular}
\end{table}

\subsubsection{Element based AMG}
\label{sec:background:spectral:amge}

The idea of using finite element stiffness matrices and associated SPSD splitting for AMG was first proposed in \cite{Brezina.19994c}. There, they solve local eigenvalue problems based on element stiffness matrices restricted to a nonoverlapping partition of DOFs. Specifically, they enrich the tentative interpolation in smoothed aggregation by identifying the smallest eigenmodes of \eqref{eq:rayleigh} with $\widetilde{M}_i \coloneqq I_{\omega_i}(I - Q_i)I_{\omega_i}$, where $I_{\omega_i}$ restricts to the nonoverlapping DOFs and $Q_i$ is a projection onto the current range of $P$ restricted to this aggregate. This identifies low-energy modes not currently captured by interpolation. Element-based AMG was then proposed in \cite{Brezina.2001}, with close follow-up \cite{Henson.2001}. The basic idea was to use element matrices to define an SPSD splitting and local energies, without relying on classical AMG assumptions like algebraically smooth error varying slowly in the direction of strong connections. In the first AMGe papers, interpolation methods are built around local energy minimization principles. This was later extended to the spectral AMGe class of methods \cite{Chartier.2003,Chartier.2007,Brezina.2011,Marques.2016}, which define coarse spaces as the span of local low-energy generalized eigenmodes.

Suppose the global operator $A$ arises from a conforming finite-element assembly:
\[
  A \;=\; \sum_{e\in\mathcal{T}_h} A_e,
\]
with local element matrices $A_e \geq 0$. Following \cite{Jones.2001}, partition the fine mesh (or DOF graph) into a set of nonoverlapping \emph{agglomerates} $\{\Omega_i\}_{i=1}^N$, which correspond to connected unions of elements. Let $\SR_i:V\to V_i$ be defined as in the DD setting \eqref{eq:SR} to restrict global vectors to the discrete DOFs of $\Omega_i$ and $\SR_i^T$ inject them back. Here we have a slight difference in notation, as $\Omega_i$ is a collection of agglomerated fine mesh \emph{elements}, but the discrete DOFs associated with $\Omega_i$ overlap with the DOFs of other agglomerations $\Omega_j,i\neq j$ at cell interfaces. The local stiffness matrix on $\Omega_i$ is then given by
\begin{equation}\label{eq:el-Atilde}
  \widetilde{A}_i \coloneqq \sum_{e\subset\Omega_i} R_i A_e R_i^T,
\end{equation}
which is SPSD on $V_i$. Note that $\widetilde{A}_i$ is assembled from the element contributions within $\Omega_i$, and thus differs from the principal submatrix $\SR_i^TA\SR_i$, which would include assembled interface couplings.

The first paper on spectral AMGe \cite{Chartier.2003} assumes $A_i$ has been symmetrically scaled to have unit diagonal. In the context of \eqref{eq:rayleigh}, this is equivalent to $\widetilde{M}_i=D_i$ given by the diagonal of $\widetilde{A}_i$, and the method uses smallest eigenmodes solved on the full agglomerate stiffness matrices (including interior and interface DOFs) with Neumann-type energy to define the global columns of interpolation. If DOFs are shared across multiple agglomerates, the contributions (in the corresponding rows of interpolation) are normalized relative to contributions from all agglomerates. In \cite{Chartier.2007}, a modified formulation is proposed which identifies \emph{minimal intersection sets} based on mesh information (faces, edges, or vertices) shared by neighboring agglomerates. For each intersection $\sigma$, a Schur complement $S_\sigma$ is formed by eliminating the DOFs interior to the union of agglomerates touching $\sigma$, and a small eigenproblem
$S_\sigma y = \lambda M_\sigma y$ is solved on the DOFs of $\sigma$. In \cite{Chartier.2007}, eigenfunctions are only used to select coarse DOFs, and interpolation is constructed using local energy minimization as in \cite{Jones.2001}. In \cite{Lashuk.2008}, a modified formulation is suggested that performs a harmonic extension of intersection set eigenmodes into the interior of element agglomerates (for specific form see \eqref{eq:extension-interior} in \Cref{sec:background:spectral:dd}).

More recent forms of AMGe \cite{Brezina.2012vh,Marques.2016} have moved away from the minimal intersecting sets and appeal to the simpler original form of \cite{Brezina.19994c}, where we decompose DOFs of $\Omega_i$ into \emph{interior} $\omega_i$ and
\emph{interface} (trace) $\Gamma_i$ as in \eqref{eq:local-gevp}. In \cite{Brezina.2012vh}, two methods are proposed to form a block-diagonal (non-overlapping) \emph{tentative} interpolation over nonoverlapping interior DOFs, which is then smoothed as in classical smoothed aggregation to expand the sparsity pattern. The first method selects the smallest eigenvectors of \eqref{eq:rayleigh} with $\widetilde{M}_i=D_i$, and then restricts these local eigenvectors by value to interior DOFs. The second method eliminates the interface first by setting $\widetilde{M}_i=\widehat{D}$, where $\widehat{D}_{\ell\ell} = 0$ if $\ell\in\Gamma_i$ and $\widehat{D}_{\ell\ell} = (D_i)_{\ell\ell}$ if $\ell\in\omega_i$. In the context of \eqref{eq:local-gevp}, this takes the form 
\begin{equation}\label{eq:local-gevp-interior}
    \begin{bmatrix}
        A_{\omega_i\omega_i} & A_{\omega_i\Gamma_i}\\[2pt]
        A_{\Gamma_i \omega_i} & A_{\Gamma_i\Gamma_i}
    \end{bmatrix}
    \begin{bmatrix} \bu_{\omega} \\ \bu_{\Gamma} \end{bmatrix}
    = 
    \lambda 
    \begin{bmatrix}
        A_{\omega_i\omega_i} & \mathbf{0}\\
        \mathbf{0} & \mathbf{0}
    \end{bmatrix}
    \begin{bmatrix} \bu_{\omega} \\ \bu_{\Gamma} \end{bmatrix},
\end{equation}
or equivalently a reduced Schur-complement-on-the-interior eigenvalue problem
\begin{equation}
    S_{\omega_i}\bv_{\omega_i} = \lambda D_{\omega_i}\bv_{\omega_i},\quad\textnormal{ where } \quad
    S_{\omega_i} \coloneqq A_{\omega_i\omega_i} - A_{\omega_i\Gamma_i}A_{\Gamma_i\Gamma_i}^{-1}A_{\Gamma_i \omega_i}.
\end{equation}
Note that \eqref{eq:local-gevp-interior} implicitly defines a harmonic extension to the interface/overlap for each eigenmode, where
\begin{equation}\label{eq:extension-overlap}
    \bu_{\Gamma_i} = -A_{\Gamma_i\Gamma_i}^{-1}A_{\Gamma_i \omega_i}\bv_{\omega_i}.
\end{equation}
This extension is not directly used in \cite{Brezina.2012vh,Marques.2016}, but will come up in the following section on GMsFEMs.

We emphasize that the original \cite{Brezina.19994c} and most recent forms \cite{Brezina.2012vh,Marques.2016} of spectral AMG use element stiffness matrices and solve local eigenproblems of the form in \eqref{eq:rayleigh}, where $\widetilde{M}_i$ is simply the diagonal of the matrix. This makes sense from the perspective of local energy minimization, as these methods utilize pointwise smoothers, which are spectrally equivalent to the diagonal of the matrix.

\subsubsection{Generalized multiscale finite elements}
\label{sec:background:spectral:GMsFEM}

The same spectral energy minimization principle underlying AMGe was developed slightly later in the multiscale finite element community as the Generalized Multiscale Finite Element Method (GMsFEM). Here, we see that these methods and the corresponding literature share significant overlap with spectral AMGe, and is built around analogous spectral problems and weighted Rayleigh quotient minimization. This has a strong relation to the theory of optimal approximation in generalized finite element methods from \cite{Babuska.2011}. 

Building on the classical multiscale FEM~\cite{Hou.1997}, in ~\cite{Efendiev.2011,Efendiev.2013} local spectral problems on overlapping coarse neighborhoods are introduced to identify low–energy modes of a heterogeneous elliptic operator and define so-called multiscale basis functions. Each local basis function $\phi_{i,k}$ is a minimizer of a continuous Rayleigh quotient as in \eqref{eq:rayleigh}, with $\widetilde{A}_i$ given by the local bilinear form \eqref{eq:el-Atilde} and $\widetilde{M}_i$ given by a local mass matrix (potentially weighted with a projection of heterogeneous coefficients). The lowest eigenmodes of these local problems define multiscale basis functions of the coarse space, analogous to the low–energy subspaces used in spectral AMGe. Consistent with being motivated on the level of finite element approximation, overlapping coarse bases are stitched together using partition of unity (PoU) functions. GMsFEMs were designed as a multiscale approximation technique, but the resulting approximation is a Galerkin projection of the fine discretization and can be used in the context of two-level solvers as well \cite{Efendiev.2013,Efendiev.2012}. Such an approach predated the formal GMsFEM \cite{Efendiev.2010} (by the same authors), and has been successfully used recently for extremely anisotropic problems as motivating this work \cite{vasilyeva2025multiscale}. 

The fundamental difference between GMsFEM methods and variations in spectral AMGe methods lies in the support of local subdomains defining coarse basis functions. With respect to the finite element mesh, spectral AMGe methods agglomerate a set of fine elements, and local subdomains defining the generalized eigenvalue problems only overlap on shared facet DOFs of neighboring agglomerates. In contrast, GMsFEMs define a coarse \emph{mesh}, and define local eigenvalue problems on ``oversampled'' subdomains, typically with respect to the support of overlapping coarse basis functions. Thus spectral AMGe local subdomains overlap with one layer of \emph{fine} DOFs shared with neighboring subdomains, while GMsFEM local subdomains overlap with one layer of \emph{coarse} DOFs shared with neighboring subdomains, which likely contains many layers of fine DOFs. Because each coarse neighborhood contains many fine elements, the oversampled eigenproblems capture long-range correlations and yield contrast-independent approximations. The result is more powerful approximation properties of GMsFEMs compared with spectral AMGe methods, while also introducing additional computational costs and challenges. This includes obvious costs like the need to solve larger local eigenvalue problems, as well as more subtle challenges such as the difficulty of (i) multilevel extensions due to significant fill-in of Galerkin coarse grids, and (ii) parallel implementations, when local subdomains will typically cover all DOFs owned by multiple MPI processes. More recently, a two-stage constraint–energy–minimizing form of
GMsFEM was proposed in~\cite{Chung.2018}, where eigenvalue problems are first solved restricted to coarse \emph{elements}. This is more or less identical to solving over the interior of agglomerates as in spectral AMGe methods discussed previously \cite{Brezina.2012vh} and shown in \eqref{eq:local-gevp-interior}. Then, basis functions are extended to the oversampled (overlapping) subdomains via energy-minimizing harmonic extensions as defined in \eqref{eq:extension-overlap}, similar in principle to the smoothing used in spectral AMGe \cite{Brezina.2012vh}, and analogous to the harmonic extensions to \emph{interior} DOFs from trace eigenvalue problems also arising in AMGe literature \cite{Lashuk.2008}. 

\subsubsection{Domain decomposition and GenEO methods}
\label{sec:background:spectral:dd}

The same local spectral and energy--minimization principles that form the basis for spectral AMGe and GMsFEM have also been developed in the DD community. The underlying goal is to construct coarse spaces for additive Schwarz or BDDC/FETI preconditioners that make the resulting two-level methods robust with respect to coefficient heterogeneity and domain partitioning. These methods, culminating in the \emph{generalized eigenproblems in the overlaps} (GenEO) framework, can be interpreted as the DD analogue of the overlapping spectral methods used in GMsFEM and the local Rayleigh--quotient minimizations of
spectral AMGe.

Classical two--level additive Schwarz methods \cite{Toselli.2005} yield convergence bounds (in theory and in practice) that depend on coefficient variation and overlap size, e.g. \cite{Pechstein.2013,Graham.2007}. For highly heterogeneous problems, this constant can grow without bound, making the preconditioner ineffective. This led to the use of local spectral problems to define coarse space operators in DD methods, e.g. \cite{Galvis.2010,Efendiev.2012,Scheichl.2011,Nataf.2011}, that yield convergence independent of heterogeneity. These approaches were formalized in the generalized eigenvalues in the overlap (GenEO) framework \cite{Spillane.2014,Spillane.2025}, where the coarse space is built from the eigenvectors of generalized eigenproblems in the overlapping or interface regions of subdomains, $\Gamma_i$. Let $\{\Omega_i\}$ denote an overlapping decomposition of the domain, with local bilinear forms $a_{\Omega_i}(\cdot,\cdot)$ corresponding to the restriction of the global stiffness form $a(\cdot,\cdot)$, and let $\PoU_i$ be diagonal PoU weights satisfying $\sum_i \SR_i^T \PoU_i \SR_i = I$. On each subdomain $\Omega_i$, GenEO defines the local spectral problem
\begin{equation}\label{eq:geneo-eig}
  a_{\Omega_i}(u,v)
  \;=\;
  \lambda\, a_{\Gamma_i}(\PoU_i u,\,\PoU_i v),
  \qquad \forall v\in V(\Omega_i),
\end{equation}
where $a_{\Gamma_i}(\cdot,\cdot)$ measures the energy in the overlapping/interface region $\Gamma_i$. The eigenvectors associated with the smallest eigenvalues are retained and multiplied by $\PoU_i$ to enforce continuity across overlaps. The span of all such weighted local modes defines the global coarse space $V_H = \sum_i \SR_i^T \PoU_i V_{H,i}$. This coarse space yields a condition number bound independent of both coefficient contrast and overlap size.

Equation~\eqref{eq:geneo-eig} is directly analogous to the local Rayleigh quotient minimization in \eqref{eq:rayleigh} underlying spectral AMGe and GMsFEM, with $\widetilde M_i$ here replaced by the PoU--weighted overlap energy form $a_{\Gamma_i}(\PoU_i\cdot,\PoU_i\cdot)$. In the discrete setting, let $\widetilde{A}_i$ denote the \emph{local} assembled stiffness matrix over $\Omega_i$ with natural/Neumann boundaries analogous to \eqref{eq:el-Atilde}, and $\widehat{A}_{\Gamma_i,\Gamma_i}$ denote a local assembled stiffness matrix over $\Gamma_i$, incorporating potentially functional $\PoU_i$. 
Then in the context of \eqref{eq:rayleigh}, \eqref{eq:geneo-eig} takes the form
\begin{equation}\label{eq:local-gevp-overlap}
    \begin{bmatrix}
        \widetilde{A}_{\omega_i\omega_i} & \widetilde{A}_{\omega_i\Gamma_i}\\[2pt]
        \widetilde{A}_{\Gamma_i \omega_i} & \widetilde{A}_{\Gamma_i\Gamma_i}
    \end{bmatrix}
    \begin{bmatrix} \bu_{\omega_i} \\ \bu_{\Gamma_i} \end{bmatrix}
    = 
    \lambda 
    \begin{bmatrix}
        \mathbf{0} & \mathbf{0}\\
        \mathbf{0} & \widehat{A}_{\Gamma_i\Gamma_i}
    \end{bmatrix}
    \begin{bmatrix} \bu_{\omega} \\ \bu_{\Gamma} \end{bmatrix},
\end{equation}
or equivalently a reduced Schur complement on the overlap eigenvalue problem
\begin{equation}\label{eq:extension-interior}
    S_{\Gamma_i}\bv_{\Gamma_i} = \bu_{\Gamma_i} = \lambda \widehat{A}_{\Gamma_i\Gamma_i}\bu_{\Gamma_i}
    \quad\textnormal{ where }\quad
    S_{\Gamma_i}\coloneqq \widetilde{A}_{\Gamma_i\Gamma_i} - \widetilde{A}_{\Gamma_i\Gamma_i}\widetilde{A}_{\omega_i\omega_i}^{-1}\widetilde{A}_{\omega_i\Gamma_i} 
\end{equation}
with harmonic (energy minimizing) extension from the overlap/interface $\Gamma_i$ to the interior $\omega_i$ naturally imposed via
\[
  \bu_{\omega_i} = -\widetilde{A}_{\omega_i\omega_i}^{-1}\textbf{}_{\omega_i\Gamma_i}\bu_{\Gamma_i}.
\]
Each eigenvector $\bu_\Gamma$ represents a low-energy mode on the interface, and its harmonic extension fills the subdomain interior with the minimum possible energy relative to $a_{\Omega_i}$. The collection of these lifted modes, globally connected through the PoU weights $\PoU_i$, defines the GenEO coarse space. If we replace the stiffness matrix on the right-hand side with a mass matrix or diagonal, this formulation coincides algebraically with a certain spectral AMGe trace eigenproblem under the identification
of overlap DOFs with subdomain interfaces \cite{Lashuk.2008}, and has a duality with later spectral AMGe formulations \cite{Brezina.2012vh} and particularly constrained GMsFEM \cite{Chung.2018}, which solve an interior eigenvalue problem \eqref{eq:local-gevp-interior} and harmonically extend to the overlap \eqref{eq:extension-overlap}. However, here we have a weighted/restricted \emph{stiffness} matrix on the right-hand side of the GEVP -- in the context of algebraic two-level theory, this makes perfect sense as we are now using an overlapping Schwarz relaxation weighted by PoU rather than pointwise relaxation that is spectrally equivalent to the diagonal, and thus $\widetilde{M}_i$ should reflect this. 

While~\cref{eq:extension-interior} was originally proposed in the first presentation of GenEO~\cite{Spillane.2014}, other generalized eigenvalue problems have been presented in the spectral domain decomposition literature, see e.g., \cite{nataf15,Bastian.2022}.

\subsubsection{Algebraic spectral coarse spaces}
In many scenarios, one may not have access to element matrices required for {\it analytic} spectral domain decomposition or algebraic multigrid methods. For instance, the discretisation kernel may not easily provide such information or the matrix does not even stem from a discretisation of a PDE. Algebraic spectral domain decomposition methods emerged from the need for a general framework to design multilevel domain decomposition methods that would, under certain conditions, fall back to {\it analytic} spectral domain decomposition methods.

A fully algebraic framework was first proposed in~\cite{al2019class} where SPSD splitting matrices were introduced along with their properties and how they play a crucial role in developing spectral coarse spaces. Also see \cite{gouarin2024fully}. Once the SPSD splitting matrices are available, the theory and practice of algebraic and analytic spectral domain decomposition follow the same path. This framework proved its success in extending the two-level GenEO method to a multilevel one~\cite{Daas.2021} and designing spectral domain decomposition method for the sparse normal equation matrix \cite{aldaas:2022aa}, diagonally dominant matrices \cite{Daas.2023}, and general matrices~\cite{aldaas25}.

\section{Spectral DD-AMG}
\label{sec:amg-dd}

\subsection{A multilevel method for $A=G^TG$}

Having defined domain decomposition smoothers for a general matrix $A$, we now utilize the special structure of $A=G^TG$ to define a two-level and naturally multilevel algorithm utilizing the domain decomposition smoothers as defined in \eqref{eq:ras} and \eqref{eq:ras-t} on each level. We emphasize that the overarching objective here is to develop an additive decomposition of the matrix $A$ based on overlapping aggregates, wherein each additive term is SPSD. This ensures that the global spectral properties of $A$ are well-represented by local approximation on overlapping aggregates, and our multigrid transfer operators are then constructed based on certain local generalized eigenvalue problems. The two-level algorithm largely follows the development in \cite{al2019class}, and here we propose a new way to facilitate multilevel recursion. The key is ensuring that the coarse grid operator also has a sparse-least squares form, $A_2 = G_2^TG_2$, via appropriate construction of sparse coarse basis functions.

\subsubsection{SPSD splitting and two-level method}

We begin by presenting the two-level algorithm based on a new algebraic SPSD splitting and following the algebraic overlapping Schwarz framework from \cite{al2019class}. An aggregation algorithm is first applied to $A$ to construct $n_c$ nonoverlapping aggregates $\{\omega_i\}_{i=1}^{n_c}$. Following the procedure described in \Cref{sec:background:dd}, we obtain the overlapping aggregates $\{\Omega_i\}_{i=1}^{n_c}$. Note, we associate the aggregates constructed on $A$ with the columns in $G$. Using the permutation matrix $\SP_i$ defined in~\cref{eq:permutation_matrix}, the matrix $G$ can be ordered as
\begin{equation}
    \label{eq:column_permuted_G}
    G \SP_i^T = \begin{pmatrix}
     G_{\omega_{i}} & G_{\Gamma_{i}} & G_{\Delta_{i}}   
    \end{pmatrix}.
\end{equation}
We now consider permuting the rows in $G$ such that we first have the rows containing nonzeros in $G_{\omega_{i}}$ followed by all other rows that are zero in $\omega_i$. We denote the set of indices of the nonzero and zero rows as $\mathrm{nz}_{i}$ and $\mathrm{z}_i$, respectively. Define the permutation matrix $Q_i = I([\mathrm{nz}_i,\mathrm{z}_i],:)$, where $[\mathrm{nz}_i,\mathrm{z}_i]$ is a vector reordering of rows of the identity. Applying $Q_i$ to $G\SP_i^T$, we arrive at the reordered matrix with the following sparsity structure:
\begin{equation}
    \label{eq:permuted_G}
    Q_i G \SP_i^T = \begin{pmatrix}
     G_{\mathrm{nz}_{i},\omega_{i}} & G_{\mathrm{nz}_{i},\Gamma_{i}} & \mathbf{0} \\
                 \mathbf{0} & G_{\mathrm{z}_i,\Gamma_{i}} & G_{\mathrm{z}_i,\Delta_{i}}   
    \end{pmatrix}.
\end{equation}
Therefore, we have
\begin{subequations}
\begin{align}
    \label{eq:GtG}
    \SP_i A \SP_i^T & = \SP_iG^TQ_i^TQ_iG \SP_i^T \\ 
    & =\begin{pmatrix}
        G_{\mathrm{nz}_{i},\omega_{i}}^T G_{\mathrm{nz}_{i},\omega_{i}} & G_{\mathrm{nz}_{i},\omega_{i}}^T G_{\mathrm{nz}_{i},\Gamma_{i}} & \mathbf{0} \\
        G_{\mathrm{nz}_{i},\Gamma_{i}}^T G_{\mathrm{nz}_{i},\omega_{i}} & G_{\mathrm{nz}_{i},\Gamma_{i}}^T G_{\mathrm{nz}_{i},\Gamma_{i}} & \mathbf{0} \\
        \mathbf{0} & \mathbf{0} & \mathbf{0}
    \end{pmatrix} +
    \begin{pmatrix}
        \mathbf{0} & \mathbf{0} & \mathbf{0} \\
        \mathbf{0} &G_{\mathrm{z}_i,\Gamma_{i}}^T G_{\mathrm{z}_i,\Gamma_{i}} & G_{\mathrm{z}_i,\Gamma_{i}}^T G_{\mathrm{z}_i,\Delta_{i}}\\
        \mathbf{0} & G_{\mathrm{z}_i,\Delta_{i}}^T G_{\mathrm{z}_i,\Gamma_{i}} & G_{\mathrm{z}_i,\Delta_{i}}^TG_{\mathrm{z}_i,\Delta_{i}}
    \end{pmatrix}
    \\
    & =\begin{pmatrix}
        G_{\mathrm{nz}_{i},\omega_{i}}^T \\
        G_{\mathrm{nz}_{i},\Gamma_{i}}^T \\
        \mathbf{0}
    \end{pmatrix}
    \begin{pmatrix}
        G_{\mathrm{nz}_{i},\omega_{i}} & G_{\mathrm{nz}_{i},\Gamma_{i}} & \mathbf{0}
    \end{pmatrix}
    +
    \begin{pmatrix}
        \mathbf{0} \\ G_{\mathrm{z}_i,\Gamma_{i}}^T \\  G_{\mathrm{z}_i,\Gamma_{i}}^T 
    \end{pmatrix}
    \begin{pmatrix}
        \mathbf{0} & G_{\mathrm{z}_i,\Gamma_{i}} & G_{\mathrm{z}_i,\Delta_{i}}
    \end{pmatrix}.
\end{align}
\end{subequations}

Equation \eqref{eq:GtG} shows how to split the matrix $A$ as a sum of two SPSD matrices with one of them only being nonzero in the overlapping aggregate $\Omega_i$. This falls into the broader class of SPSD splittings originally defined in \cite{al2019class}, but by nature of the least squares formulation here constructing a valid splitting is relatively simple/natural. 
Note that
\[ 
    G^TG = \sum_{\ell=1}^m \bg_\ell^T\bg_\ell,
\]
where $\bg_\ell$ is the $\ell$th row of $G$. Now consider taking the outer product in the left side of \eqref{eq:GtG} for every $\omega_\ell$, corresponding to rows of $G$ that are nonzero in columns of aggregate $\omega_\ell$. By nature of $\{\omega_\ell\}_{\ell=1}^{n_c}$ providing a nonoverlapping covering of the domain, every row of $G$ will have nonzero entries in at least one aggregate $\omega_\ell$. To that end, if we compute the local outer product for every aggregate $\omega_\ell$, we necessarily compute all elements of $A = G^TG$. Many indices $i\in[1,n]$ will correspond to nonzero rows in multiple aggregates. To ensure the summation over all local outer products equals the global matrix, we must normalize each row by the number of aggregates in which it has nonzero entries. Altogether this leads to the following result. 

\begin{lemma}
Define the multiplicity of node $j\in\{1,\ldots,m\}$ corresponding to rows of $G$ as $M(j) = |\{k\ : \ j \in \mathrm{nz}_k\}|$, that is, $M(j)$ is the number of overlapping aggregates that the $j$th row of $G$ is nonzero in. Define $W = $diag$(1/M(j)) \in \mathbb{R}^{m\times m}$ to average contributions to node $j$ from its shared overlapping aggregates, and define weighted aggregate matrices
\begin{equation}
    \label{eq:local_splitting}
    \widetilde{A}_{i} \coloneqq \begin{pmatrix}
        G_{\mathrm{nz}_{i},\omega_{i}}^T W(\mathrm{nz}_i,\mathrm{nz}_i) G_{\mathrm{nz}_{i},\omega_{i}} & G_{\mathrm{nz}_{i},\omega_{i}}^T W(\mathrm{nz}_i,\mathrm{nz}_i)G_{\mathrm{nz}_{i},\Gamma_{i}}\\
        G_{\mathrm{nz}_{i},\Gamma_{i}}^T W(\mathrm{nz}_i,\mathrm{nz}_i)G_{\mathrm{nz}_{i},\omega_{i}} & G_{\mathrm{nz}_{i},\Gamma_{i}}^T W(\mathrm{nz}_i,\mathrm{nz}_i)G_{\mathrm{nz}_{i},\Gamma_{i}}
    \end{pmatrix}.
\end{equation}
Then $\widetilde{A}_i$ is SPSD for all $i\in[1,n_c]$ and
\begin{equation}
    \label{eq:split_A}
    A = \sum_{i=1}^{n_c} \SR_i^T \widetilde{A}_{i} \SR_i.
\end{equation}
\end{lemma}
\begin{proof}
    The proof follows from the above discussion and noting that $W$ is diagonal with all diagonal entries $\geq 0$, in which case it admits a unique real-valued square root. Then $\widetilde{A}_i$ can be expressed as a symmetric block outer product for all $i$, ensuring it is SPSD. 
\end{proof}

Given this decomposition and principal submatrices as used in our smoother, we will solve local generalized eigenvalue problems involving $\widetilde{A}_i$ and use the resulting local eigenvectors to define our global transfer operators. In order to maintain sparsity, we will restrict the eigenvectors to only be nonzero on nonoverlapping subdomains. Let $\widehat{Z}_i$ denote the local eigenvector matrix for aggregate $i$. We then use the Schwarz transfer operators $\{\SR_{\omega_i}^T\}_{i=1}^{n_c}$ that map from nonoverlapping local aggregates to the global domain to define the multigrid interpolation operator as the first $n_c$ columns of local eigenvectors in $\widehat{Z}$:
\begin{equation}\label{eq:mg-p}
    \MP \coloneqq
    \begin{pmatrix}
        \SR_{\omega_1}^T\widehat{Z}_1 & ... & \SR_{\omega_{n_c}}^T \widehat{Z}_{n_c}
    \end{pmatrix}.
\end{equation}
Note that the resulting interpolation operator in \eqref{eq:mg-p} is block-diagonal by nonoverlapping aggregate, \emph{with no overlap across aggregates.}

Coarse-grid correction takes the standard form
\begin{equation}\label{eq:cgc}
    \bx_{k+1} = \bx_k + P(P^TAP)^{-1}P^T(\bb - A\bx_k),
\end{equation}
and a natural two-level method arises by coupling algebraic domain decomposition smoothers, e.g. \eqref{eq:ras}, with \eqref{eq:cgc}. Typically this would be done in a multiplicative fashion, but one can also consider an additive method
\[
    \bx_{k+1} = \bx_k + M_2^{-1} (\bb-A\bx_k) \coloneqq 
        \bx_k + (\MP (\MP^TA\MP)^{-1} \MP^T  + M_1^{-1}) (\bb-A\bx_k),
\]
for domain decomposition smoother $M_1^{-1}$. In the case that $M_1^{-1}$ correponds to an additive Schwarz method (that is, \eqref{eq:ras} without the scaling local matrices by the PoU $\PoU_i$), it was shown in \cite{al2019class} that the eigenvalues of the preconditioned operator $M_2^{-1}A$ lie in the range
\[
\sigma\left(M_2^{-1}A\right) \subset \Big[ (2+(2k_c + 1)\tau)^{-1}, k_c+1 \Big],
\]
where $k_c$ is the number of colors required to color the aggregates such that each two neighboring aggregates do not share the same color.

The Schwarz relaxation blocks are given by the overlapping subdomains $\{\Omega_i\}$ constructed during aggregation.  Each $\Omega_i$ forms one block in the RAS \eqref{eq:ras} or RAS--T \eqref{eq:ras-t} smoother; the corresponding principal submatrix $A(\Omega_i,\Omega_i)$ is extracted once during setup and reused throughout iterations.  The Boolean PoU vectors $\mathcal{D}_{i}$ ensure that, in RAS, each fine-grid DOF receives exactly one update. Our multigrid method then uses one iteration of RAS \eqref{eq:ras} as a pre-smoother and one iteration of RAS-T \eqref{eq:ras-t} as a post-smoother, which yields an SPD preconditioner when $A$ is SPD.

\subsubsection{DD local eigenvalue problem}

Appealing to the two-level theory developed in \cite{al2019class}, one way to construct transfer operators is based on solving the following local generalized eigenvalue problems over each overlapping aggregate:
\begin{equation}
    \label{eq:gevp}
    \PoU_i A_{i} \PoU_i \bu_{\Omega_i} = \lambda \widetilde{A_{i}}\bu_{\Omega_i}.
\end{equation}
For each aggregate, all eigenvectors associated with eigenvalues \emph{larger} than a user-specified tolerance $\tau$, are kept to include as part of the global interpolation operator. Note that because the generalized eigenproblem is written in the form \eqref{eq:gevp}, the roles of ``smooth'' and ``oscillatory'' modes are reversed relative to the classical  $A\bu=\lambda M\bu$ ordering. Therefore, the relevant coarse-space modes correspond to the largest generalized eigenvalues.

The PoU as defined in \eqref{eq:pou2} restricts local eigenvectors to the nonoverlapping aggregates $\{\omega_i\}$. It is also applied to the rows and columns of the left-hand side of the generalized eigenvalue problem in \eqref{eq:gevp}. As a result, we can obtain the required eigenvectors from \eqref{eq:gevp} restricted to nonoverlapping aggregates by solving reduced local Schur complement generalized eigenvalue problems. The generalized eigenvalue problem in \eqref{eq:gevp} explicitly takes the form
{\footnotesize
\begin{equation}
    \label{eq:gevp2}
    \begin{pmatrix}
        G_{\mathrm{nz}_{i},\omega_{i}}^T G_{\mathrm{nz}_{i},\omega_{i}} & \mathbf{0} \\ \mathbf{0} & \mathbf{0} 
    \end{pmatrix}
    \begin{pmatrix}
        \bu_{\omega_i} \\ \bu_{\Gamma_i}
    \end{pmatrix}
    = \lambda 
    \begin{pmatrix}
        G_{\mathrm{nz}_{i},\omega_{i}}^T W(\mathrm{nz}_i,\mathrm{nz}_i) G_{\mathrm{nz}_{i},\omega_{i}} & G_{\mathrm{nz}_{i},\omega_{i}}^T W(\mathrm{nz}_i,\mathrm{nz}_i)G_{\mathrm{nz}_{i},\Gamma_{i}}\\
        G_{\mathrm{nz}_{i},\Gamma_{i}}^T W(\mathrm{nz}_i,\mathrm{nz}_i)G_{\mathrm{nz}_{i},\omega_{i}} & G_{\mathrm{nz}_{i},\Gamma_{i}}^T W(\mathrm{nz}_i,\mathrm{nz}_i)G_{\mathrm{nz}_{i},\Gamma_{i}}
    \end{pmatrix}
    \begin{pmatrix}
        \bu_{\omega_i} \\ \bu_{\Gamma_i}
    \end{pmatrix},
\end{equation}
}where only generalized eigenvectors restricted to $\omega_i$, $\bu_{\omega_i}$, are used in constructing interpolation \eqref{eq:mg-p}. Note that such eigenvectors associated with nonzero eigenvalues can equivalently be obtained via a Schur complement generalized eigenvalue problem,
\begin{align}
    \label{eq:gevp-S}
\begin{split}
        & G_{\mathrm{nz}_{i},\omega_{i}}^T G_{\mathrm{nz}_{i},\omega_{i}} 
        \bu_{\omega_i} 
    = \lambda \widetilde{S}_{\omega_i}\bu_{\omega_i}, \quad\textnormal{where} \\
    &\widetilde{S}_{\omega_i} \coloneqq 
        G_{\mathrm{nz}_{i},\omega_{i}}^T W(\mathrm{nz}_i,\mathrm{nz}_i) G_{\mathrm{nz}_{i},\omega_{i}} - \\
        &\hspace{4ex}G_{\mathrm{nz}_{i},\omega_{i}}^T W(\mathrm{nz}_i,\mathrm{nz}_i)G_{\mathrm{nz}_{i},\Gamma_{i}} 
        \left(G_{\mathrm{nz}_{i},\Gamma_{i}}^T W(\mathrm{nz}_i,\mathrm{nz}_i)G_{\mathrm{nz}_{i},\Gamma_{i}}\right)^{-1}
        G_{\mathrm{nz}_{i},\Gamma_{i}}^T W(\mathrm{nz}_i,\mathrm{nz}_i)G_{\mathrm{nz}_{i},\omega_{i}}.
\end{split}
\end{align}
Thus we choose the largest eigenvalues of the reduced generalized eigenvalue problem in \eqref{eq:gevp-S}. 

\begin{remark}
    Starting from the block $2\times 2$ operators $A_i$ and $\widetilde{A}_i$, one could consider the full generalized eigenvalue problem (without $\mathcal{D}_i$), or six variations that result in a natural reduced eigenvalue problem over $\omega_i$, due to different application of $\mathcal{D}_i$ on the left and/or right of each operator. From an AMGe perspective, some of the formulations other than \eqref{eq:gevp} may seem appealing/more natural, but tests have indicated that no other reduced eigenvector formulation works well. Simple tests indicated a growth in iteration count of $5-10\times$ purely due to modifying the local generalized eigenvalue problem.
\end{remark}

\subsubsection{Multilevel method} 

In \cite{Daas.2021} a strategy is proposed to generalize the construction of an algebraic two-level method to further levels relying on the fact that the sum of the Galerkin projection (on the coarse space in a two-level method) of local SPSD matrices is itself a local SPSD splitting of the coarse space operator, subject to a conveniently chosen aggregation strategy. While it facilitated the constructing of more levels, the restriction on choosing aggregates at the coarse level can yield load imbalance or result in hindered performance of the multilevel method if the aggregation process was not consistent across all levels. 

In this article, we design a simple technique to extend the two-level method proposed in \cite{Daas.2021} to an arbitrary number of levels with complete freedom in forming aggregates at each level without affecting the performance of the multilevel method.
The technique is simple, and relies on the structure of the coarse space matrix resulting from the matrix $A=G^TG$.
Since $P$ is sparse, $G_2 = GP$ is also sparse, and the coarse space matrix $A_2 = P^TG^TGP = G_2^TG_2$ is itself in the same product structure as the system matrix in \cref{eq:Ax=b}.
Hence, the same procedure can be repeated to construct a two-level method for $A_2$.

\subsection{Algorithmic discussion}
\label{sec:amg-dd:alg}

An outline of the algorithm is provided in \Cref{alg:lsamgdd}, and we provide additional details and commentary in this section.

\begin{algorithm}[t]
\caption{Multilevel LS--AMG--DD Construction}
\label{alg:lsamgdd}
\begin{algorithmic}[1]

\REQUIRE Least--squares operator $A_0 = G_0^T G_0$, maximum levels $L_{\max}$,
maximum coarse size $n_{\mathrm{coarse}}$, aggregation passes $n_{\mathrm{agg}}$,
target condition number $\kappa$, and minimum coarsening ratio $c_{\min}$.

\STATE Initialize level index $\ell = 0$.

\WHILE{$\ell < L_{\max}-1$ and $\dim(A_\ell) > n_{\mathrm{coarse}}$}

    \STATE 
    \textbf{Aggregation.} Construct $n_{\mathrm{agg}}$ passes of aggregation to obtain a
    nonoverlapping\par partition $\{\omega_i^{(\ell)}\}$ of fine-grid DOFs.
    Add unaggregated nodes to neighboring \par aggregates.

    \STATE \textbf{Overlapping subdomains.}
    For each aggregate $i$, define an overlapping set\par
    $\Omega_i^{(\ell)} = \omega_i^{(\ell)} \cup$ (neighbors in $A_\ell$),
    and a Boolean partition-of-unity vector $\mathcal{D}_{i,\ell}$ \par that is
    $1$ on $\omega_i^{(\ell)}$ and $0$ on $\Omega_i^{(\ell)}\!\setminus\!\omega_i^{(\ell)}$.

    \STATE \textbf{Smoothing.} Define each overlapping subdomain $\Omega_i^{(\ell)}$ as a block for the RAS \par and RAS–T smoothers, using $A_\ell(\Omega_i^{(\ell)},\Omega_i^{(\ell)}$ as the local operator.

    \STATE \textbf{Local SPSD splitting.}
    Use rows of $G_\ell$ intersecting $\Omega_i^{(\ell)}$ to build 
    local SPSD \par matrix $\widetilde{A}_{i,\ell}$ that satisfies
    $
        A_\ell = \sum_{i=1}^{N_c} \SR_{i,\ell}^T \, \widetilde{A}_{i,\ell} \,\SR_{i,\ell},
    $
    where $\SR_{i,\ell}$ restricts to \par$\Omega_i^{(\ell)}$.

    \STATE \textbf{Local generalized eigenproblems.}
    Restrict all operators to nonoverlapping \par DOFs $\omega_i^{(\ell)}$
    and solve the reduced local GEP
    $
        A_{\omega\omega} \, u = \lambda \, S_{\omega} \, u,
    $
    where $S_\omega$ is the \par Schur complement of $\widetilde{A}_{i,\ell}$.
    Select all eigenvectors $\lambda > \tau_\ell$ derived from $\kappa$ up \par to the maximum number per aggregate $|\omega_i|/c_{\min}$.

    \STATE \textbf{Interpolation (coarse basis).}
    Assemble block-diagonal interpolation
    \par $
        P_\ell = 
        \begin{bmatrix}
            R_{\omega_1}^T \widehat{Z}_{1,\ell} &
            \cdots &
            R_{\omega_{N_c}}^T \widehat{Z}_{N_c,\ell}
        \end{bmatrix},
    $
    where $\widehat{Z}_{i,\ell}$ contains selected eigenvectors \par on $\omega_i^{(\ell)}$.

    \STATE \textbf{Galerkin projection and LS propagation.}
    Form next-level operators \par
    $
        G_{\ell+1} = G_\ell P_\ell$ and $
        A_{\ell+1} = P_\ell^T A_\ell P_\ell = G_{\ell+1}^T G_{\ell+1}.
    $

    \STATE Increase the level: $\ell \gets \ell + 1$.

\ENDWHILE

\STATE \textbf{return} Multilevel hierarchy $\{A_\ell, G_\ell, P_\ell\}$.
\end{algorithmic}
\end{algorithm}

\subsection{Aggregation}
A main component of this framework is to use AMG aggregation techniques to coarsen relatively slowly and facilitate multilevel hierarchies, rather than full domain decomposition approaches like Metis. One is tempted to further use more advanced AMG aggregation techniques and SOC measures to define ``good'' aggregates, however this does not appear to be useful in practice. Similar observations were made in early AMGe work \cite{Chartier.2003}. More advanced aggregation routines almost always use some form of SOC to do the aggregation. As we have discussed though, for harder problems this is exactly one of the issues with AMG methods, that error may \emph{not} vary smoothly in the direction of strong connections, making SOC measures ill-suited to defining good aggregates. Moreover, the role of the local spectral problems is precisely to capture a good coarse grid, and our experience has been trying to ``help'' this process with advanced SOC or aggregation causes more harm than good in terms of worse convergence and/or larger complexities. To that end, here we use exclusively the ``standard'' aggregation routine as implemented in PyAMG based directly on matrix entries rather than SOC. We have additionally added a simple routine to assign any unaggregated fine nodes (rows of the aggregation matrix with all zeros) to one of the aggregates of their neighbors in the matrix. This is relatively rare, but necessary to ensure all DOFs are represented in the overlapping aggregates defining the Schwarz-based smoother. 

In addition, as we will see later there are problems where convergence cannot be obtained with the slower coarsening provided by single-pass AMG aggregation. Effectively, the sum over selected local eigenvectors mapped to the global domain does not adequately capture the global near null space. This can be remedied by constructing larger aggregates, which yields local generalized eigenvalue problems that span a larger portion of the domain. We do this by performing multiple passes of aggregation, effectively aggregating aggregates. Specifically, we form a tentative interpolation operator $T_1$ of zeros and ones from the aggregation matrix, and tentative coarse-grid $T_1^TAT_1$. We then repeat the aggregation process to form a second tentative interpolation operator $T_2$, and define our aggregation matrix via $T_1T_2$. 

\subsection{Coarsening and eigenvectors}
The algebraic-DD methods motivating this work have rigorous theory that relate the condition number of the preconditioned operator to local eigenvalues, which provides an automated mechanism to choose local eigenvectors via specifying target condition number $\kappa$. We have found this relation is less robust in the multilevel setting, at least in terms of being an accurate predictor of convergence. Thus for this paper, we set a fixed $\kappa = 50$ indicating we take all local modes that are relatively ill-conditioned, but limit to a specified coarsening ratio per aggregate. Specifically, a local eigenvalue threshold is specified from $\kappa$ from the formulae in \cite{al2019class}
\[
    \mathrm{thresh} \coloneqq \max \left\{ 0.1, \frac{\kappa - n_{\textnormal{color}}}{n_{\textnormal{color}}n_{\textnormal{multiplicity}}} \right\}.
\]
Then, if we set a minimum coarsening ratio of three, we will take all eigenvectors with associated eigenvalue larger than $\mathrm{thresh}$ up to a third of the size of the aggregate. To reduce complexity, it can help to coarsen faster on certain levels than others, in particular we have found one can coarsen faster on coarser levels in the hierarchy. To that end, we further consider level-specific coarsening, where for example LS$_{3,4}$ denotes a least-squares solver that coarsens by a minimum factor of three (per aggreagte) on the first level and four on all subsequent levels.

\subsection{Implementation}
Our LS-AMG-DD method has been implemented in PyAMG \cite{bell2023pyamg}. The implementation is moderately efficient in terms of performing the most expensive computations in C++ either through the PyAMG backend or using numpy/scipy. However, a number of components such as the local spectral problems are looped over in python, in this case to call the scipy eigensolver, which can be significantly slower than compiled C++ looping. We also construct dense inverse decompositions for the smoother. Performance implementations of algebraic-DD like the HPDDM library \cite{jolivet2021ksphpddm} use sparse Cholesky decompositions for the inverse of subdomain operators, making their setup significantly cheaper than a dense $\mathcal{O}(|\Omega_i|)^3$ operations, and subsequent application and solve $\mathcal{O}(|\Omega_i|)$ operations rather than dense $2|\Omega_i|^2$ operations. For PDE problems, constructing a sparse Cholesky factorization or preconditioner can also be significantly faster than a direct inverse, e.g., \cite{schaefer2021sparse}. To that end, in numerical results we focus on convergence rates and operator complexities. Future work will consider a performant implementation as in, e.g., HPDDM, which is particularly relevant on modern computing architectures and GPUs, where dense linear algebra is extremely fast. 

\section{Numerical results}
\label{sec:numerics}

\subsection{(Anisotropic) Laplacian}
\label{sec:numerics:laplacian}

We begin by studying variations in rotated anisotropic diffusion 
\begin{align}
    -\nabla \cdot K \nabla u & = f, \quad  u\in\Omega, \\
    u & = g, \quad u\in\partial\Omega,
\end{align}
where
\begin{equation}
    Q(\theta) =
    \begin{pmatrix}
    \cosa & -\sina\\
    \sina & \phantom{-}\cosa
    \end{pmatrix}, \quad
    D = 
    \begin{pmatrix}
    \epsilon & 0\\[2pt] 0 & 1
    \end{pmatrix}, \quad
    K = Q D Q^T ,
\end{equation}
for rotation angle $\theta$ and anisotropy ratio $\epsilon$. Such an equation has been widely used as a representative ``hard'' elliptic problem in AMG literature. We formulate a specific tensor grid finite difference discretization such that the assembled linear system can be expressed in the least squares form \eqref{eq:Ax=b}. In particular, on a uniform grid with spacings $h_x, h_y$, define the forward differences
$(D_x^+ u)_{i,j} \coloneqq \tfrac{u_{i+1,j}-u_{i,j}}{h_x}$ and
$(D_y^+ u)_{i,j} \coloneqq \tfrac{u_{i,j+1}-u_{i,j}}{h_y},$
and let the backward differences be the (negative) adjoints $D_x^- \coloneqq -(D_x^+)^T $ and $D_y^- \coloneqq -(D_y^+)^T .$ Then the discrete divergence is $\nabla^- \!\cdot \mathbf{v} = D_x^- v_x + D_y^- v_y,$ and the discrete forward gradient (in Matlab notation) is $\nabla^+ u = \begin{pmatrix}D_x^+ u ; D_y^+ u\end{pmatrix}.$ Note that $K = B B^T $ with $B := Q \sqrt{D}$ and $\sqrt{D}=\mathrm{diag}(\sqrt{\epsilon},\,1)$. Then define the discrete operators
\begin{equation}
G \coloneq B^T  \nabla^+, \qquad
G^T  \coloneqq -\,\nabla^- \!\cdot\, B,
\end{equation}
so that
\begin{equation}
G^T  G\,u = -\,\nabla^- \!\cdot\!\left( B\,B^T \,\nabla^+ u \right)
= -\,\nabla^- \!\cdot\!\left( K\,\nabla^+ u \right).
\end{equation}
Note that for $\theta=0$ and $\epsilon = 1$ this reduces to the classic 5-point finite-difference stencil for the Laplacian. 

We begin by fixing a non-grid-aligned angle of $\theta = \pi/6$ and considering anisotropy ratio from $\epsilon = 1$ (isotropic) to $\epsilon = 10^{-7}$ on a $500\times 500$ spatial grid, resulting in 250,000 total DOFs. \Cref{fig:epsilon_scaling} shows the number of iterations to reduce the residual by a factor of ten (computed based on average convergence factor solving system to $10^{-8}$ relative residual) for each solver, as well as the corresponding operator complexity of the solver. We see that while classical AMG and SA degrade significantly as anisotropy ratio $\epsilon$ decreases, several forms of the LS-AMG-DD solver provide excellent convergence invariant to anisotropy. The operator complexity of the robust solvers is larger than AMG, on the order of 5-6, but this is more than compensated for by the robust convergence. We also emphasize that in the isotropic setting, one can form LS-AMG-DD solvers that coarsen slightly more aggressively (e.g., $LS_4$ or $LS_{3,4}$ that obtain operator complexities only slightly larger than AMG, as well as AMG convergence. Although, these methods are not likely to be competitive with AMG in practice due to higher setup and cycle cost, the fact that they produce comparable convergence and operator complexity as classical AMG methods, but can be naturally extended to be robust for arbitrary anisotropy is critical for robust and efficient black-box AMG solvers.
\begin{figure}[!htb]
  \centering
  \begin{subfigure}[t]{0.475\textwidth}
    \centering
    \includegraphics[width=\textwidth]{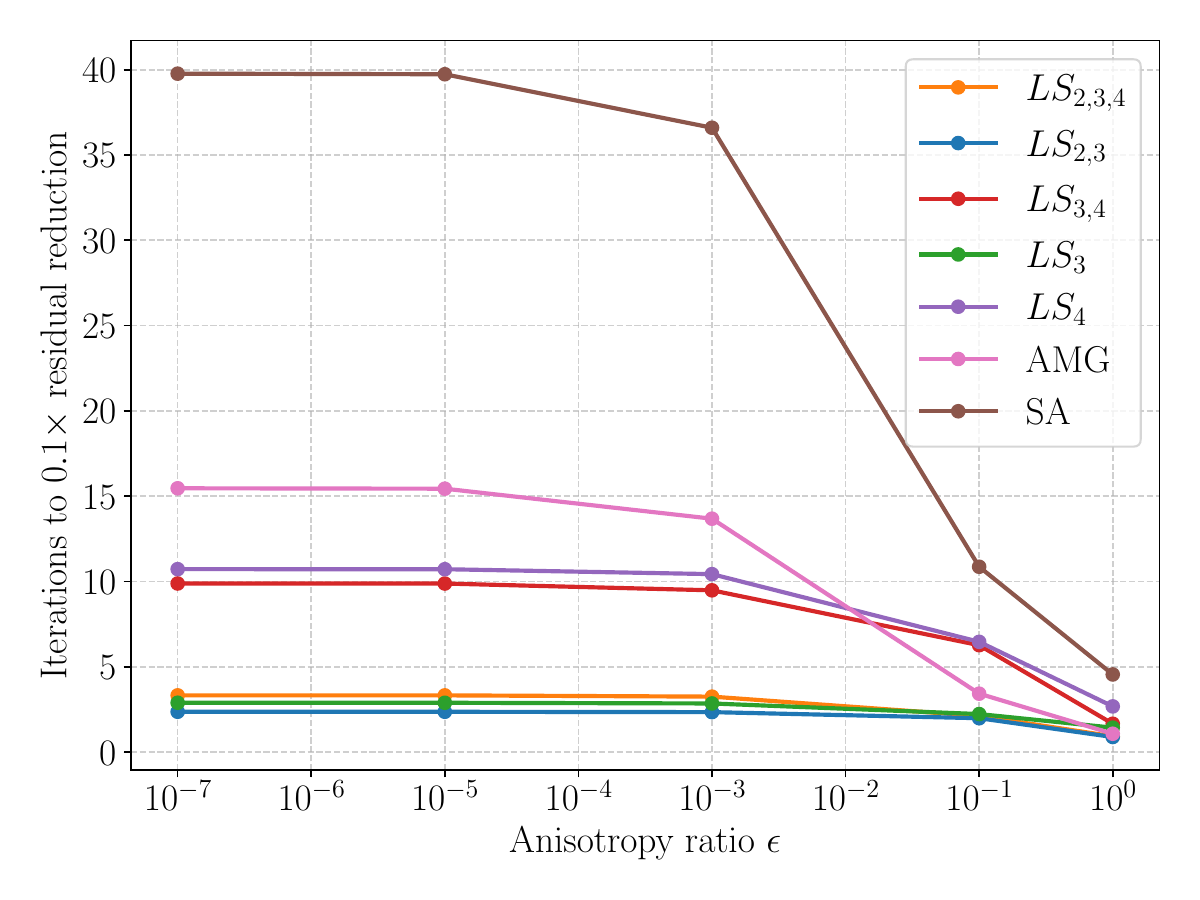}
    \caption{Iterations to 0.1$\times$ residual reduction vs. anisotropy coefficient $\epsilon$}
    \label{fig:iters_vs_epsilon}
  \end{subfigure}
  \begin{subfigure}[t]{0.475\textwidth}
    \centering
    \includegraphics[width=\textwidth]{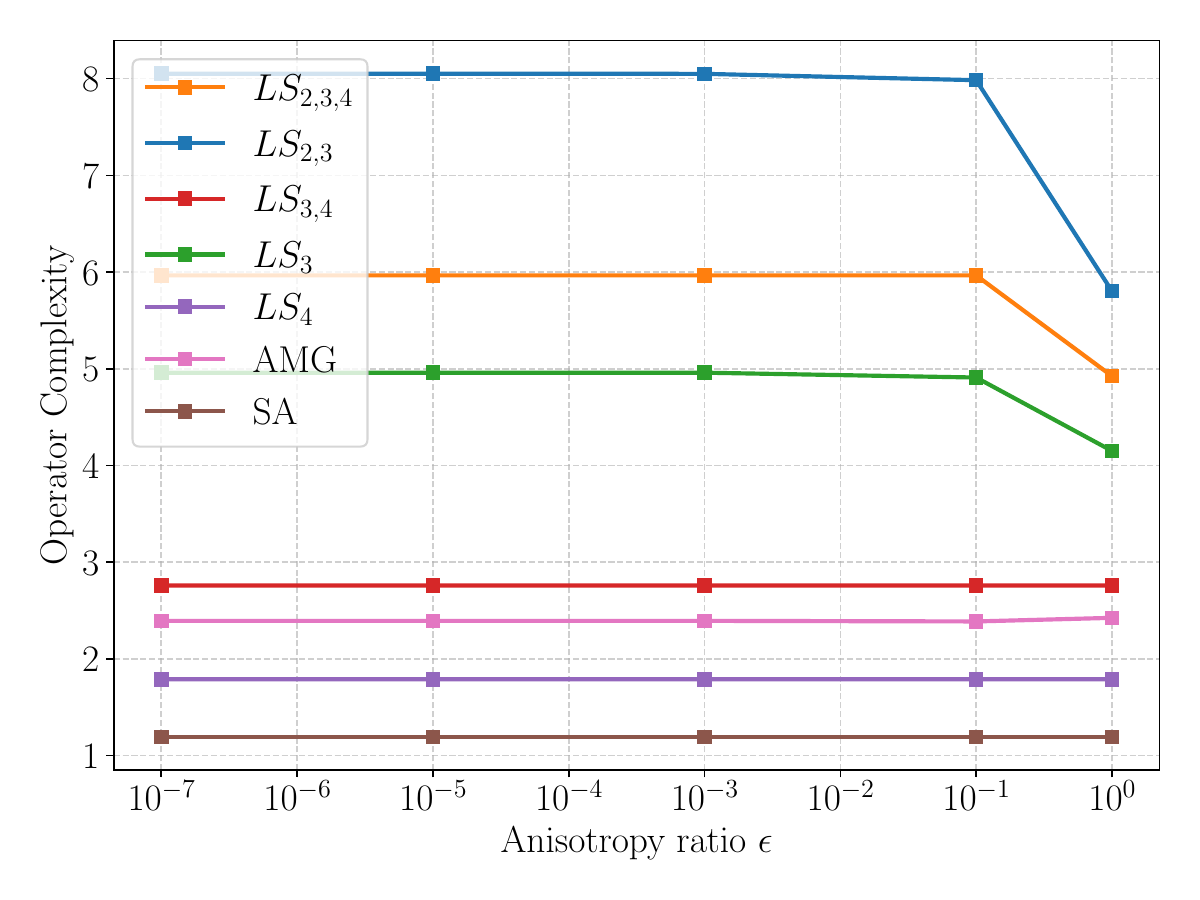}
    \caption{Operator complexity vs. anisotropy coefficient $\epsilon$}
    \label{fig:OC_vs_epsilon}
  \end{subfigure}
  \caption{Solver performance as a function of anisotropy coefficient $\epsilon$ for fixed $N=250,000$ and $\theta = \pi/6$. For reference, the average convergence factor at $\epsilon = 10^{-7}$ for $LS_{2,3,4}$, AMG, and SA are (respectively) $\rho = 0.5, 0.86, 0.94$.}
  \label{fig:epsilon_scaling}
\end{figure}

Next we pick a fixed anisotropy of $\epsilon = 10^{-5}$ and angle $\theta = \pi/6$, and consider scaling in problem size. Analogous results for convergence and operator complexity are shown in \Cref{fig:size_scaling}. As in the case of anisotropy ratio, we see that $LS_{2,3}$, $LS_3$, and $LS_{2,3,4}$ demonstrate perfectly scalable convergence with respect to problem size up to a $1000\times 1000$ grid with 1M DOFs. This is compensated with a slight growth in operator complexity with increase in $N$, but for $LS_{2,3,4}$ this growth is slow and manageable. In contrast, for AMG and SA we see significant growth in iteration count as problem size increases.
\begin{figure}[!htb]
  \centering
  \begin{subfigure}[t]{0.475\textwidth}
    \centering
    \includegraphics[width=\textwidth]{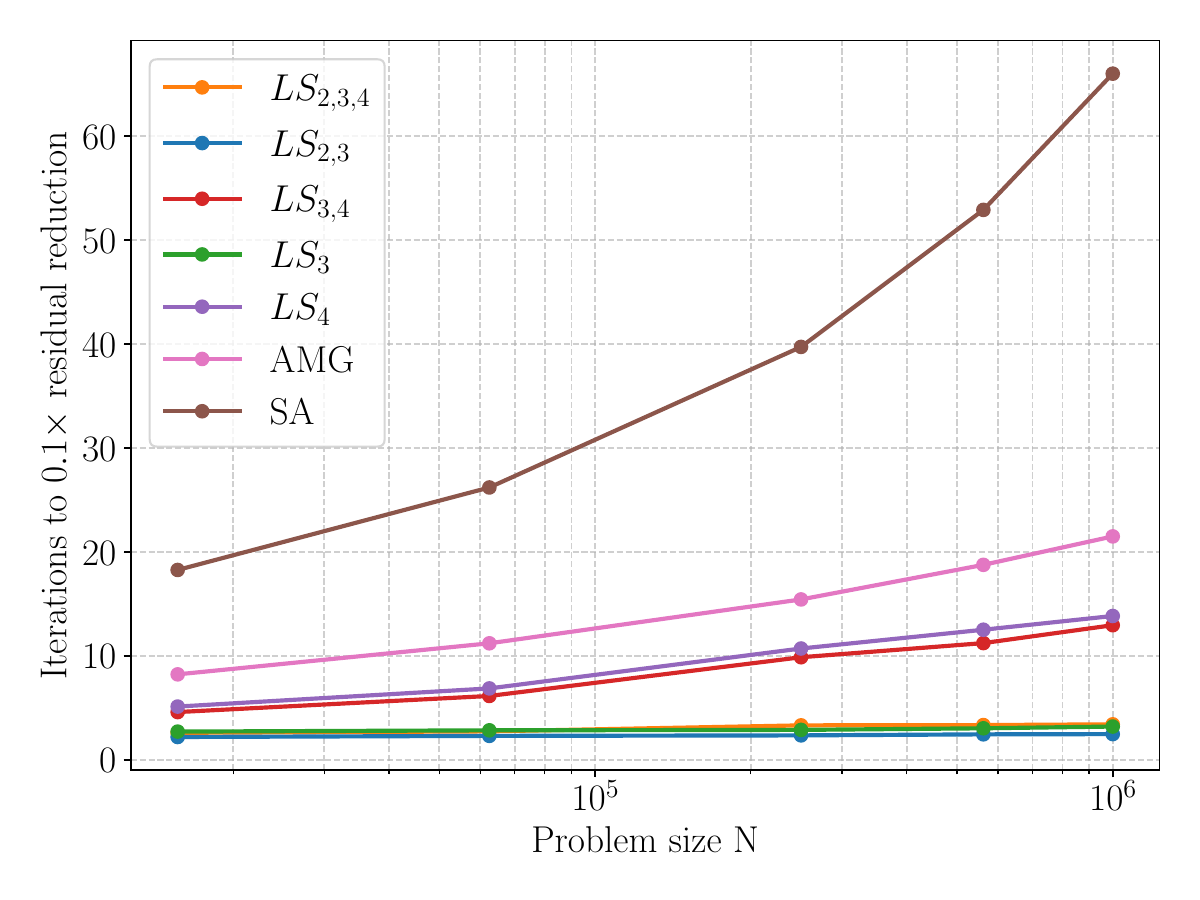}
    \caption{Iterations to 0.1$\times$ residual reduction vs. problem size $N$}
    \label{fig:iters_vs_N}
  \end{subfigure}
  \begin{subfigure}[t]{0.475\textwidth}
    \centering
    \includegraphics[width=\textwidth]{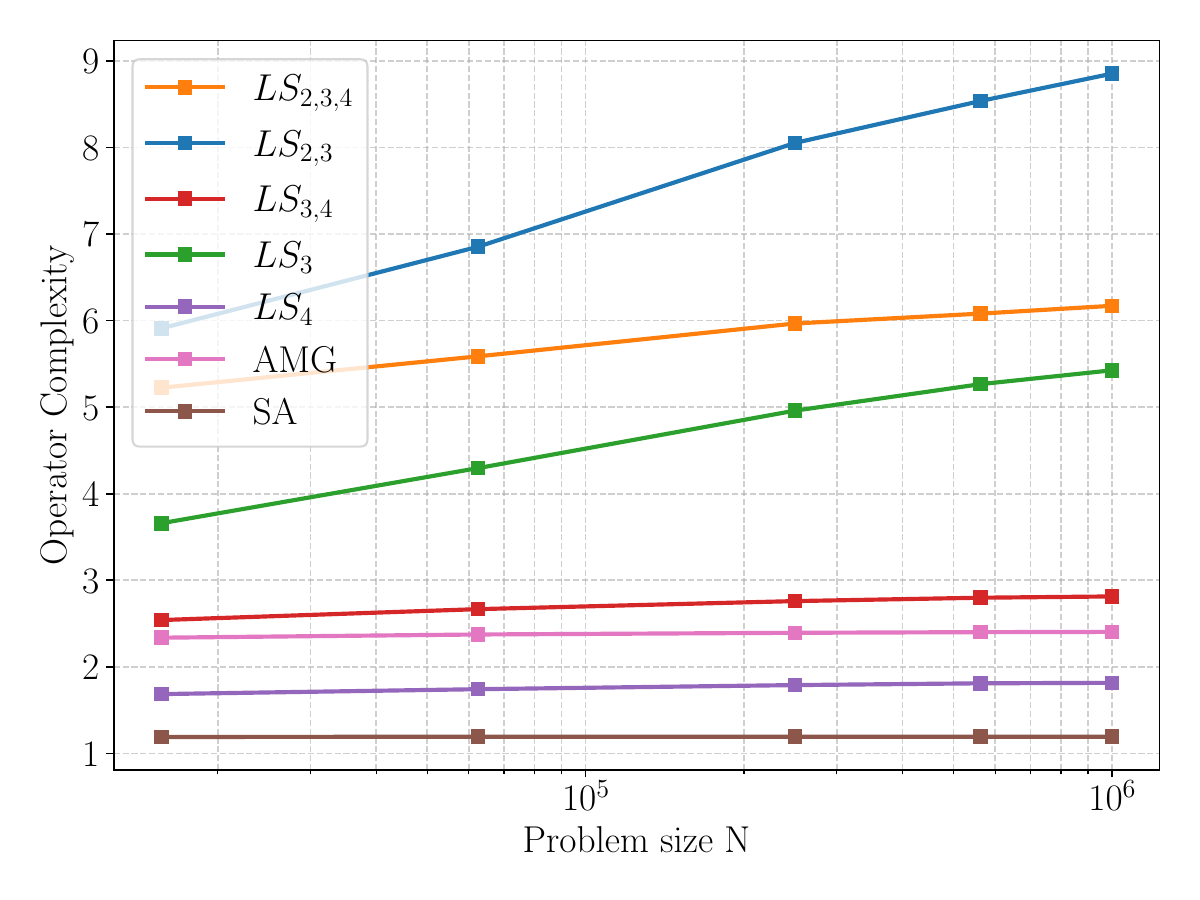}
    \caption{Operator complexity vs. problem size $N$}
    \label{fig:OC_vs_N}
  \end{subfigure}
  \caption{Solver performance as a function of total DOFs $N$ for fixed $\epsilon=10^{-5}$ and $\theta = \pi/6$. For reference, the average convergence factor at $\epsilon = 10^{-5}$ for $LS_{2,3,4}$, AMG, and SA are (respectively) $\rho = 0.51, 0.9, 0.97$.}
  \label{fig:size_scaling}
\end{figure}

Last, we identify $LS_{2,3,4}$ as a good mix of robust and scalable, with manageable operator complexity, and compare convergence of SA, AMG, and $LS_{2,3,4}$ as a function of anisotropy angle $\theta \in(0,\pi/2)$ for fixed $N =250,000$ and $\epsilon = 10^{-5}$ in \Cref{fig:angle_scaling}. We see that convergence of $LS_{2,3,4}$ is almost uniform across all angles, with surprisingly some degradation near the grid-aligned case of $\theta = 0$. In this particular region, classical AMG also performs quite well, because for truly grid-aligned anisotropy the traditional AMG assumptions of algebraically smooth error varying slowly in the direction of strong connections holds. 

\subsection{Anisotropic heat conduction in magnetic confinement fusion}
\label{sec:numerics:aniso}

We now consider extremely anisotropic diffusion equations as arise in magnetic
confinement fusion simulations. These problems are significantly more challenging than
those in the previous section for a number of reasons, and were a primary motivation for this work. Anisotropy is spatially varying and time evolving, aligned with the magnetic field $\mathbf{B}$, and with anisotropy ratios that can reach $10^{10}$ or higher \cite{hoelzl2021jorek,wimmer2024accurate}. In magnetic confinement fusion, the disparity in diffusion coefficient $\kappa_\| \gg \kappa_\perp$, for flow parallel and perpendicular to the magnetic field, respectively, represents heat flowing rapidly along magnetic field lines but very slowly across them.The presence of closed field lines yields topologically nontrivial null spaces and makes the implicit system extremely ill-conditioned (and ill-posed without a time derivative), often beyond the reach of standard AMG. Last, significant care has to be taken in discretizing the equations to avoid spurious heat loss perpendicular to the magnetic field. Indeed, for scenarios in which the mesh cannot feasibly be aligned with the field lines -- as is the case for many types of magnetohydrodynamic instabilities with complex magnetic field configurations -- continuous or discontinuous primal discretizations result in
\begin{wrapfigure}{r}{0.5\textwidth} %
\centering
\vspace{-2ex}
  \includegraphics[width=0.5\textwidth]{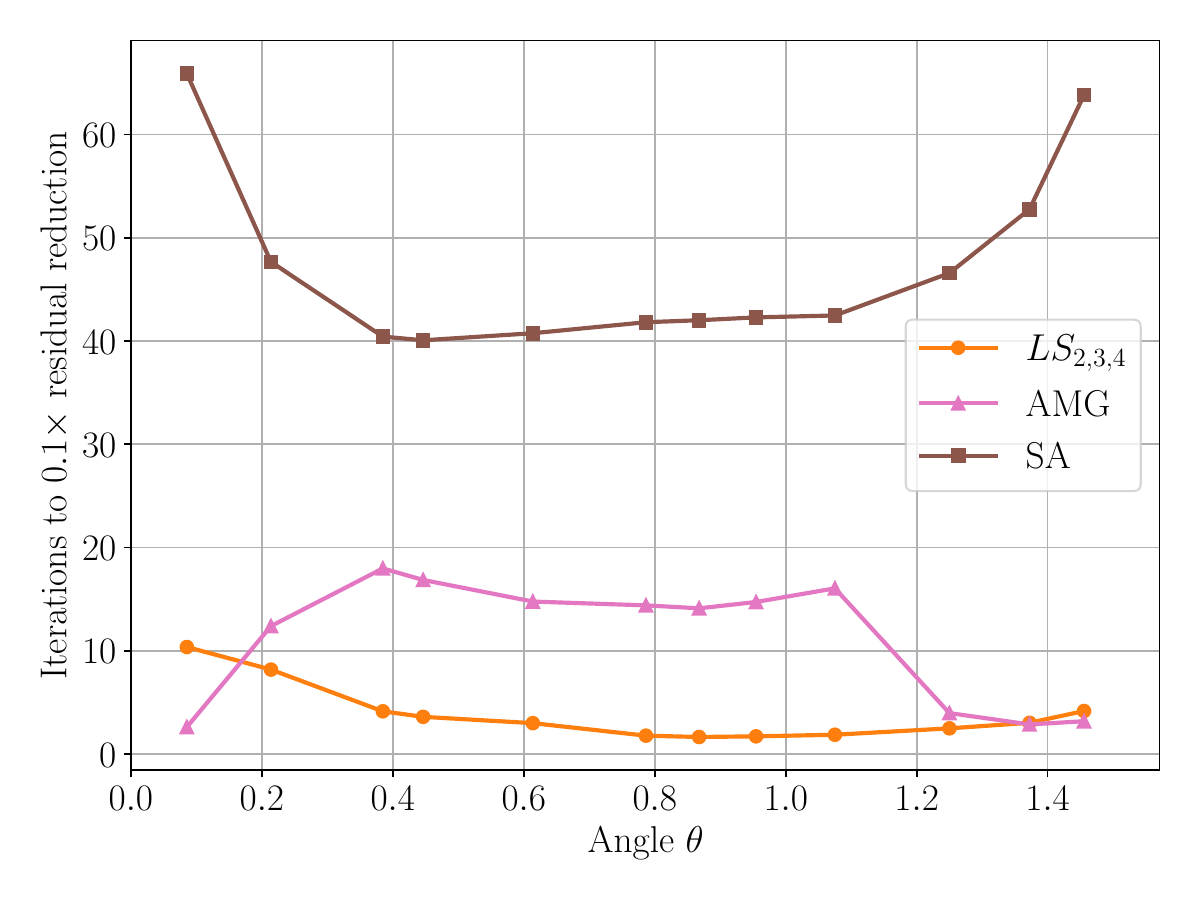}
  \caption{Solver performance as a function of angle $\theta$ for fixed total DOFs $250,000$ and $\epsilon=10^{-5}$.}
  \label{fig:angle_scaling}
  \vspace{-3ex}
\end{wrapfigure}
nontrivial heat loss, and are not sufficient for accurate and realistic simulations. In previous work of ours \cite{wimmer2024fast}, we considered a mixed formulation based on one proposed in \cite{gunter2007finite}, which introduces an auxiliary variable $\zeta_h$ corresponding to the directional gradient $\mathbf{b} \cdot \nabla T_h$, for temperature field $T_h$ and unit-length magnetic field $\mathbf{b} = \mathbf{B}/|\mathbf{B}|$. We then used a discontinuous Galerkin-based upwind stabilized form for the directional gradients, which, next to improved accuracy, in particular led to an efficient solver procedure when using multigrid methods targeting hyperbolic equations \cite{manteuffel2018nonsymmetric}. However, this solver approach relied on magnetic field configurations with open field lines only, which is not realistic for magnetic confinement fusion simulations. For a naive $H^1$ primal discretization, we have also found two-level solvers based on GMsFEMs to be effective \cite{vasilyeva2025multiscale}, indicating the potential of using local spectral problems.

In this work, we consider the mixed form from \cite{gunter2007finite}. With care this discretization can offer good discretization accuracy, but is simpler than the upwind stabilized form of \cite{wimmer2024fast} and our continuous Galerkin-based extension \cite{wimmer2024accurate}, which lead to non-symmetric matrices that we will consider in future work. We consider quadrilateral meshes here, because \cite{gunter2007finite} is accurate on such meshes but we have observed significant degradation in accuracy for triangular meshes \cite{wimmer2024fast}. The temperature is posed in the $k$th polynomial order continuous Galerkin space $\mathbb{V}_T$, corresponding to $Q_k$ equipped with suitable Dirichlet boundary conditions. Further, the auxiliary variable is set in the discontinuous Galerkin space $\mathbb{V}_\zeta = dQ_{k-1}$. We then solve for $(T_h, \zeta_h) \in (\mathbb{V}_T, \mathbb{V}_\zeta)$ such that
\begingroup
\begin{subequations}
\begin{align}
&\left\langle \eta, \frac{\partial T}{\partial t}\right\rangle + \langle \eta, \sqrt{\kappa_\Delta} \mathbf{b} \cdot \nabla T \rangle + \langle \nabla \eta, \kappa_\perp \nabla T \rangle = 0 &\forall \eta \in \mathring{\mathbb{V}}_T, \\
&\langle \phi, \zeta_h \rangle = \langle \phi, \sqrt{\kappa_\Delta} \mathbf{b} \cdot \nabla T \rangle &\forall \phi \in \mathbb{V}_\zeta,
&
\end{align}
\end{subequations}
\endgroup
for parallel and perpendicular conductivity coefficients $\kappa_\parallel$ and $\kappa_\perp$, respectively, and $\kappa_\Delta = \kappa_\parallel - \kappa_\perp$. Additionally, $\mathring{\mathbb{V}}_T$ corresponds to $\mathbb{V}_T$ equipped with homogeneous boundary conditions. If we let $M_T$ and $M_\zeta$ denote mass matrices, $L$ the discrete (isotropic) Laplacian, $G_b$ the discrete scalar-form transport operator $\mathbf{b} \cdot \nabla T$, the linearized system associated with each time step takes the form
{\small
\begin{equation}\label{eq:mat-discr}
\begin{pmatrix}
  \frac{1}{\Delta t}M_T + \kappa_\perp L && \sqrt{\kappa_\Delta} G_b^T \\
  -\sqrt{\kappa_\Delta}G_b && M_\zeta
\end{pmatrix}.
\end{equation}
Eliminating the auxiliary variable for a Schur complement in temperature, we have
\begin{equation}\label{eq:approx-S}
S_T\coloneqq 
  \frac{1}{\Delta t}M_T + \kappa_\perp L + \sqrt{\kappa_\Delta} G_b^TM_\zeta^{-1} \sqrt{\kappa_\Delta}G_b \approx \begin{pmatrix} D_T^{1/2} & \sqrt{\kappa_\Delta} G_b^TM_\zeta^{-1/2} \end{pmatrix}
  \begin{pmatrix} D_T^{1/2} \\ \sqrt{\kappa_\Delta} M_\zeta^{-1/2}G_b \end{pmatrix},
\end{equation}
}where $M_\zeta$ is either diagonal or block diagonal due to the local support of DG basis functions, thus allowing fast closed form construction of $M_\zeta^{-1/2}$, and $D_T \approx \frac{1}{\Delta t}M_T + \kappa_\perp L$ is simply the matrix diagonal. Because the auxiliary variable can be eliminated directly and Schur complement explicitly formed, we apply a block LDU preconditioner to \eqref{eq:mat-discr} wrapped with flexible GMRES and use inner CG to solve $S_T$ preconditioned by the LS-AMG-DD method proposed here built on the least-squares approximation in \eqref{eq:approx-S}. Note, this approximation is only good for small $\kappa_\perp \ll \kappa_\Delta$, but this encompasses our general regime of interest. For mixed or more isotropic regimes, one could introduce an additional variable so that $\Delta = (\nabla)^T\nabla$ is posed in a mixed least-squares sense, but such a generalization is outside the scope of this paper.

As a test problem, we consider a temperature field whose gradient is aligned with closed field lines inscribed in a 2D unit square domain $\Omega$, following \cite{sovinec2004nonlinear,gunter2007finite}. $\mathbf{B}$ and $T$ are then given by
\begingroup
\begin{subequations} \label{Nimrod_IC}
    \begin{align}
        &T(\mathbf{x}, 0) = \cos\big(\pi (x - 1/2)\big)\cos\big(\pi(y - 1/2)\big), \;\;\; T|_{\partial \Omega} = 0,\\
        &\mathbf{B}(\mathbf{x}) = \left(-\frac{\partial T(\mathbf{x}, 0)}{\partial y}, \frac{\partial T(\mathbf{x}, 0)}{\partial x}\right)^T,
    \end{align}
\end{subequations}
\endgroup
(see Figure \ref{fig_nimrod_2d}) and we fix $\kappa_\perp = 1$, while varying $\kappa_\parallel \in [10^2,10^8]$. Further, the time step is set to $\Delta t = 10^{-3}$, and we consider a quadrilateral mesh $\Delta x = 0.017$, which following \cite{wimmer2024accurate} is regular up to a small perturbation of each interior node in order to avoid fortuitous cancellations related to symmetries.

\begin{wrapfigure}{r}{0.35\textwidth} %
\centering
\includegraphics[width=0.35\textwidth]{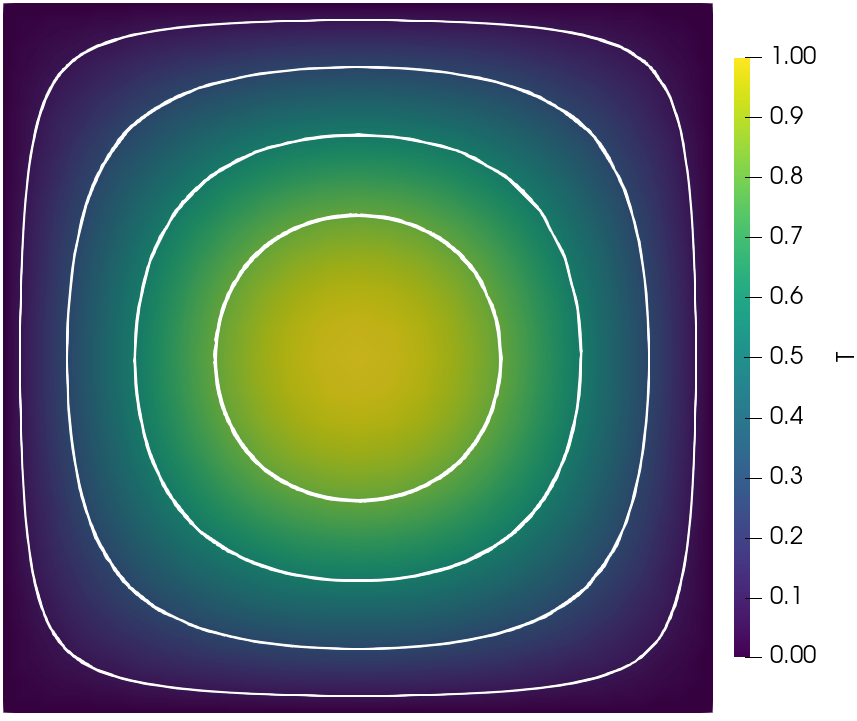}
\caption{Initial temperature field for text case \eqref{Nimrod_IC}, with fixed magnetic field lines shown in white.} \label{fig_nimrod_2d}
\vspace{-3ex}
\end{wrapfigure}

For the mixed discretization on quads, even on small problems \emph{conjugate gradient accelerated by classical AMG or smoothed aggregation preconditioning applied to $S_T\bx = \bbf$ does not reduce the residual by a factor of 0.1 in 1000 iterations.} Thus this is not just a matter of slow convergence, but a fundamental inability of existing AMG techniques to solve this problem (we have tried many other methods and tuning of parameters as available in PyAMG as well without success). To that end, we do not present results from other AMG methods, as they have all failed. In addition, the single pass of standard aggregation level used in \Cref{sec:numerics:laplacian} also fails to converge on large anisotropy. Effectively, the overlapping spectral problems are not sufficiently large and expressive to capture the near null space.

To that end, we consider two-pass aggregation based on classical aggregation in PyAMG \cite{bell2023pyamg} and discussed in \Cref{sec:amg-dd:alg}, with a coarsening ratio of $[4,5]$ meaning four on the first level and five on all subsequent levels as discussed in \Cref{sec:amg-dd:alg}. \Cref{fig:mcf_results} shows scaling studies in anisotropy ratio and problem size, while fixing the other variable. We see that the proposed LS-AMG-DD method is able to robustly solve these problems across a wide range of anisotropy ratios, all the way to $\kappa_\| / \kappa_\perp = 10^8$. Scalability in problem size is clearly suboptimal in both convergence factor and operator complexity, but here we simply seek a multilevel method that is able to solve these equations on large-scale. 

\begin{figure}[!htb]
  \centering
  \begin{subfigure}[t]{0.475\textwidth}
    \centering
    \includegraphics[width=\textwidth]{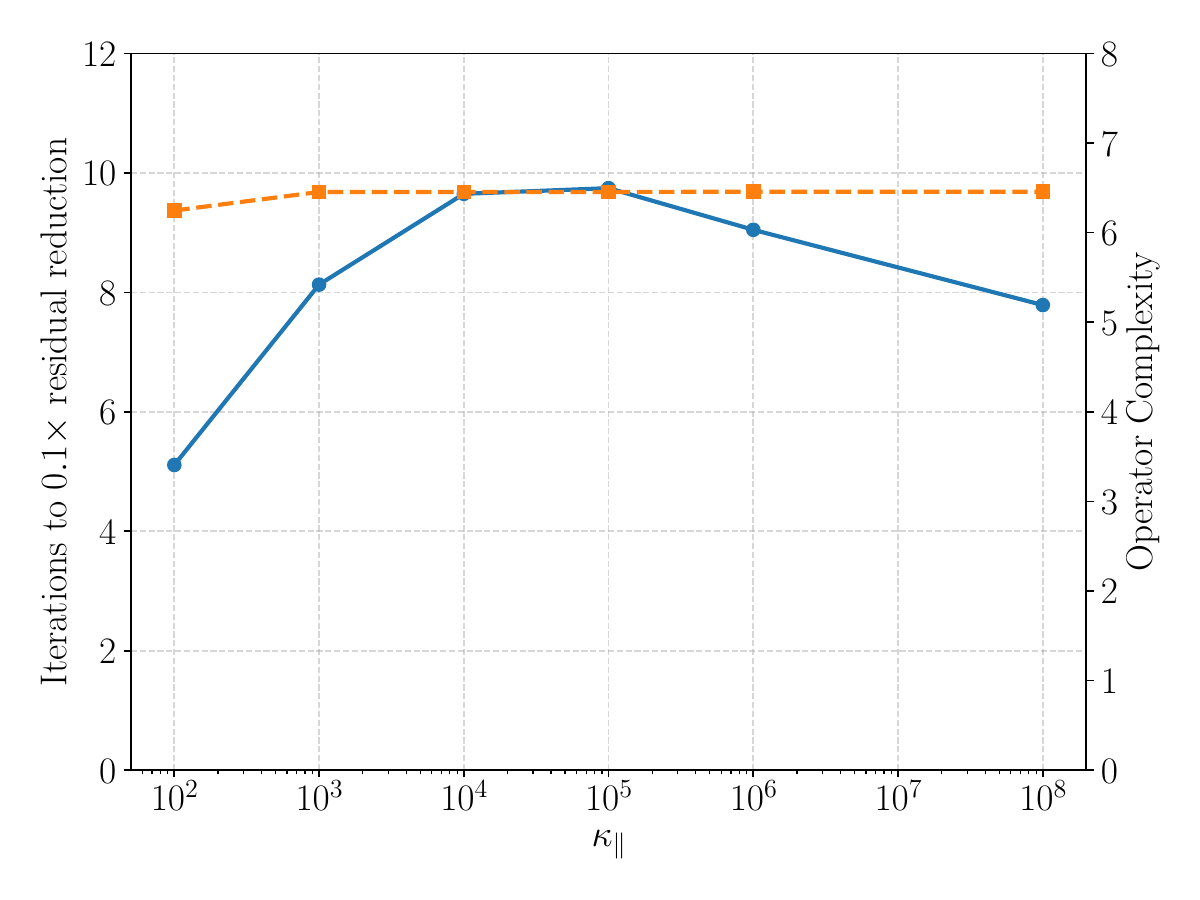}
    \caption{Iterations to 0.1$\times$ residual reduction (solid) and operator complexity (dashed) vs. parallel conductivity $\kappa_\|$.}
    \label{fig:mcf_iters_vs_k}
  \end{subfigure}
  \quad
  \begin{subfigure}[t]{0.475\textwidth}
    \centering
    \includegraphics[width=\textwidth]{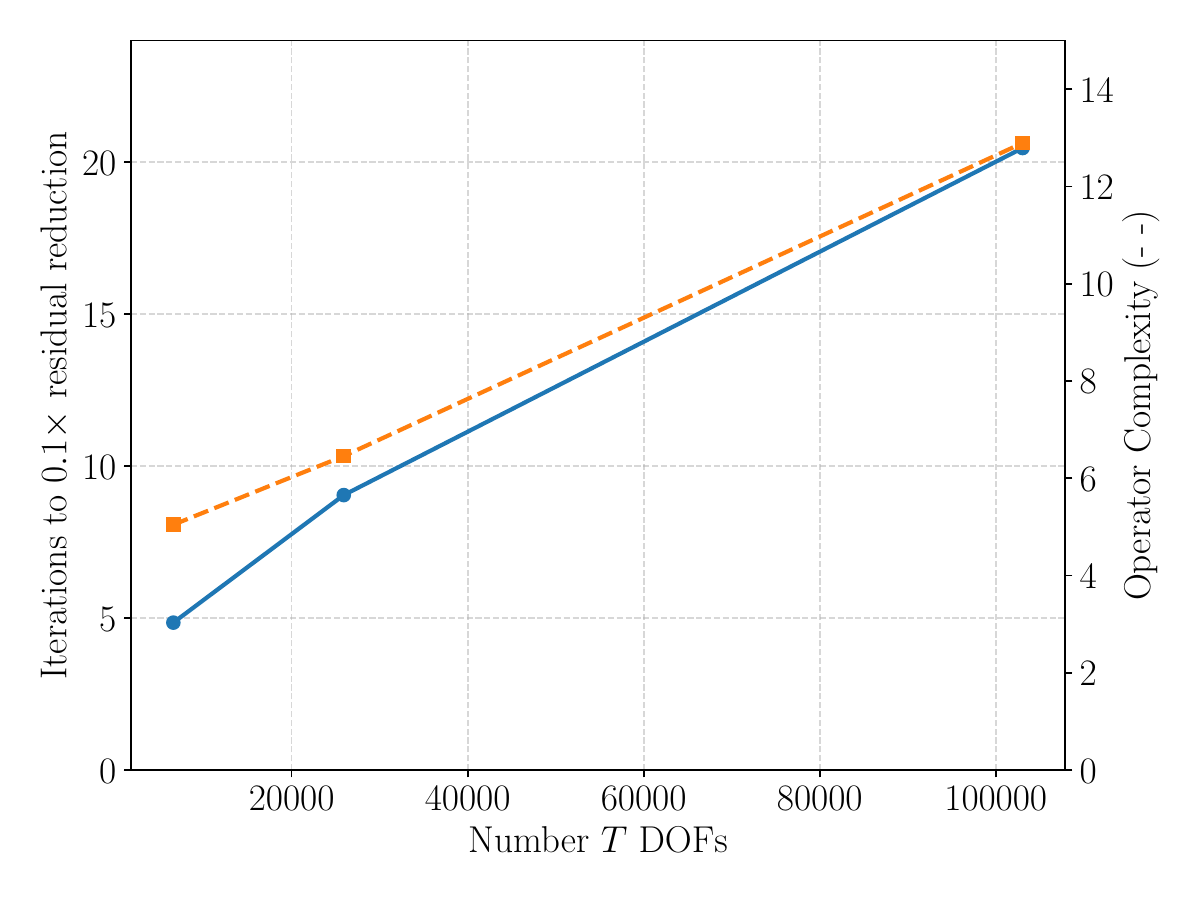}
    \caption{Iterations to 0.1$\times$ residual reduction (solid) and operator complexity (dashed) vs. problem size $N$.}
    \label{fig:mcf_iters_vs_N}
  \end{subfigure}
  \caption{Solver performance and operator complexity as a function of total DOFs $N\approx 25$K for fixed mesh refinement $\Delta x=0.017$ and varying $\kappa_\| \in[10^2, 10^8]$ (left), and fixing $\kappa_\| = 10^6$ and considering three levels of mesh refinement, $\Delta x\in\{0.008,0.017,0.034\}$. For reference, the average convergence factors range from 0.64 at $\kappa_\| = 10^2$ to $0.78$ at $\kappa_\| = 10^5$, for fixed $\kappa_\perp = 1$.}
  \label{fig:mcf_results}
  \vspace{-4ex}
\end{figure}

\section{Conclusion}
\label{sec:conclusion}

In this work we introduced a new algebraic multilevel method that combines overlapping Schwarz smoothers with locally constructed spectral coarse spaces for sparse least-squares operators of the form 
$A=G^TG$. We first review the distinct fields that spectral coarse grids arise in, namely DD, AMGe, and GMsFEMs. By exploiting the factorized structure of $A$, we develop an inexpensive SPSD splitting that enables localized generalized eigenproblems whose solutions are real-valued, definite, and define effective and sparse coarse-grid transfer operators. The resulting approach provides a fully algebraic and naturally multilevel framework that avoids global eigenproblems, allows fully algebraic coarsening, and maintains reasonable sparsity through the use of nonoverlapping supports in columns of interpolation. In particular, the method is able to coarsen slowly and maintain low operator complexities like classical AMG methods, or coarsen aggressively and use large local spectral problems to define coarse basis functions when necessary.

Numerical experiments demonstrate that the proposed LS–AMG–DD method achieves robustness and scalability across challenging problems, including highly anisotropic diffusion and extreme anisotropy arising in magnetic confinement fusion models, regimes where AMG methods often struggle or fail outright. The method delivers convergence rates largely independent of anisotropy strength, and moderate growth in complexity and iteration counts, relative to existing state of the art. For our model fusion problems, existing AMG methods are unable to reduce the residual even marginally in 1,000 preconditioned CG iterations, while the proposed methods are able to robustly solve to high accuracy. A primary topic of future work will be adaptive measures to automate the coarsening rate and selection of coarse eigenmodes, so that the algorithm only works as hard as necessary to provide a robust solve. In addition, we will consider extensions to non-least squares problems, extensions to nonsymmetric problems and relating to recent developments in nonsymmetric spectral coarse spaces, e.g. \cite{krzysik2025optimal,ali2025generalized,bootland2023overlapping}, and a more performant implementation. Overall, this work demonstrates that overlapping spectral domain decomposition ideas can be fused with algebraic multigrid coarsening to produce practical and robust multilevel solvers for a broad class of least-squares systems.

\bibliographystyle{siamplain}
\bibliography{dd-refs}

\end{document}